\documentclass[10pt,reqno]{amsart}

\usepackage{enumerate}
\usepackage{amsmath, amssymb, amsthm,amsfonts}
\usepackage{mathrsfs}
\usepackage{esint}
\usepackage{xcolor}
\usepackage{mathtools}
\usepackage{hyperref}
\usepackage{bm}
\usepackage{accents}
\usepackage[foot]{amsaddr}
\usepackage{todonotes}

\makeatletter

\numberwithin{equation}{section}

\newtheorem{thm}{Theorem}[section]
\newtheorem{theorem}[thm]{Theorem}
\newtheorem{lemma}[thm]{Lemma}
\newtheorem{corollary}[thm]{Corollary}

\newtheorem{proposition}[thm]{Proposition}
\theoremstyle{definition}

\newtheorem{example}[thm]{Example}
\newtheorem{assumption}[thm]{Assumption}
\theoremstyle{remark}

\newtheorem{remark}[thm]{\bf{Remark}}

\newcommand\aint{-\hspace{-0.38cm}\int}

\newcommand\bA{\mathbb{A}}
\newcommand\bB{\mathbb{B}}

\newcommand\bE{\mathbb{E}}

\newcommand\bH{\mathbb{H}}

\newcommand\bM{\mathbb{M}}
\newcommand\bN{\mathbb{N}}

\newcommand\bR{\mathbb{R}}

\newcommand\bZ{\mathbb{Z}}

\newcommand\cA{\mathcal{A}}

\newcommand\cF{\mathcal{F}}

\newcommand\cL{\mathcal{L}}

\newcommand\cS{\mathcal{S}}

\begin{document}

\title[]{A frequency localized maximum principle and its applications to time fractional equations with rough coefficients}

\author[J.-H. Choi]{Jae-Hwan Choi}
\address{Department of Mathematics, The Catholic University of Korea, 43 Jibong-ro, Bucheon-si, Gyeonggi-do, 14662, Republic of Korea}
\email{mathjhchoi@catholic.ac.kr}
\thanks{J.-H. Choi was supported by the National Research Foundation of Korea (NRF) grant funded by the Korea government (MSIT) (No. RS-2026-25477288).}

\author[J. Ryu]{Junhee Ryu}
\address{School of Mathematics, Korea Institute for Advanced Study, 85 Hoegi-ro, Dongdaemun-gu, Seoul, 02455, Republic of Korea}
\email{junhryu@kias.re.kr}
\thanks{J. Ryu was supported by a KIAS Individual Grant (MG101501) at Korea Institute for Advanced Study and by the National Research Foundation of Korea (NRF) grant funded by the Korea government (MSIT) (No. RS-2026-25477288).}


\begin{abstract}
We establish a frequency localized maximum principle for parabolic second-order partial differential equations with the Caputo time fractional derivative  $\partial_t^\alpha$ of order $\alpha\in(0,1)$. 
The leading coefficients are uniformly elliptic, measurable in time, and independent of space. 
We construct suitable equivalent norms on frequency localized subspaces of $L_\infty(\mathbb{R}^d)$, for which a frequency localized maximum principle holds.
The corresponding result may fail in the usual $L_\infty$ norm even for the classical heat equation.
As an application, we prove existence, uniqueness, and regularity estimates in Sobolev–Zygmund spaces with Muckenhoupt weights in time for equations with spatially variable coefficients and zero initial data.
The coefficients are strongly measurable in time with values in a H\"older–Zygmund space of positive order and uniformly bounded in this space. 
\end{abstract}

\maketitle

\section{Introduction}

The classical maximum principle plays an important role in the regularity theory of partial differential equations.
A maximum principle for Littlewood-Paley components, known as the frequency localized maximum principle, has also been applied to a variety of regularity problems.
For example, the following frequency localized maximum principle is known for sufficiently regular solutions of the fractional heat equation $\partial_t u+(-\Delta)^{\beta/2}u=f$ with $0<\beta<2$:
\begin{equation} \label{eq10100149}
\frac{d}{dt}\|\Delta_k u\|_{L_\infty(\bR^d)}+c_\beta 2^{\beta k}\|\Delta_k u\|_{L_\infty(\bR^d)} \leq\|\Delta_k f\|_{L_\infty(\bR^d)}, \quad k\geq1,
\end{equation}
almost everywhere in time, where $c_\beta>0$ is independent of $k$; see \cite{Li13add,WZ11add}.
See Section \ref{sec_main_results} for the precise definition of the Littlewood-Paley projection $\Delta_k$.

We briefly review some applications of frequency localized maximum principles, including well-posedness results for nonlinear equations, Bernstein inequalities, and regularity estimates for equations involving nonlocal operators.
Wang and Zhang~\cite{WZ11add} used a frequency localized maximum principle to prove well-posedness results for the two-dimensional quasi-geostrophic equation in critical Besov spaces.
Li~\cite{Li13add} established frequency localized Bernstein inequalities for $1<q<\infty$.
In addition, the exponential contractivity of frequency localized semigroups was established, including the case $q=\infty$ for fractional Laplacians of order $0<\beta<2$.
For more general nonlocal operators, Song and Xie~\cite{SX23} used frequency localized maximum principles to obtain H\"older estimates and well-posedness results for stochastic differential equations with singular coefficients.
Cheng, Hao, and Zhang~\cite{CHZ25} used the same method to prove Schauder estimates and averaging principles for stochastic differential equations driven by stable processes.

However, the frequency localized maximum principle does not necessarily hold for local spatial operators.
Even for the classical heat equation $\partial_t u=\Delta u+f$, the analogue of \eqref{eq10100149} with $\beta=2$ may fail.
For simplicity, we consider a one-dimensional counterexample. Fix $\omega = 5/4$ and choose $u_0(x) = a\cos(\omega x) - b\cos(3\omega x)$ so that
\begin{equation*}
\Delta_1 u_0(x) = \frac{9\cos(\omega x)-\cos(3\omega x)}{8}.
\end{equation*}
For the heat solution $u(t)=e^{t\Delta}u_0$,
\begin{equation*}
\partial_t(\Delta_1 u)(0,0) = (\Delta_1 u_0)''(0) = \frac{-9\omega^2 + (3\omega)^2}{8} = 0,
\end{equation*}
which yields $(\Delta_1u)(t,0)=1+O(t^2)$ as $t\downarrow0$. However, \eqref{eq10100149} with $\beta=2$ and $f=0$ would imply
\begin{equation*}
\|\Delta_1u(t)\|_{L_\infty(\bR)} \leq e^{-4c_2t},
\end{equation*}
which is impossible for sufficiently small $t>0$.
Li~\cite{Li13add} pointed out that the semigroup estimate with multiplicative constant one fails in the standard $L_\infty$ norm, although estimates with a larger multiplicative constant remain valid.
Such estimates do not imply the differential inequality above; see also \cite{BCD11}.

Our goal in this paper is to establish a frequency localized maximum principle for local second-order operators by a different approach.
To overcome the difficulty described above, we introduce suitable equivalent norms on frequency localized subspaces of $L_\infty(\bR^d)$.
More specifically, we construct norms $\|\cdot\|_k$, $k\geq1$, such that
\begin{equation*}
\|u\|_{L_\infty(\bR^d)} \leq \|u\|_k \leq N\|u\|_{L_\infty(\bR^d)}
\end{equation*}
whenever the Fourier transform of $u$ is supported in $\{\xi\in\bR^d:2^{k-1}\leq|\xi|\leq2^{k+1}\}$, where $N$ is independent of $k$.
In this framework, for sufficiently regular solutions of $\partial_t u=\Delta u+f$,
\begin{equation*}
\frac{d}{dt}\|\Delta_k u\|_k + \kappa 2^{2k}\|\Delta_k u\|_k \leq \|\Delta_k f\|_k, \quad k\geq1,
\end{equation*}
where $\kappa>0$.

This approach is particularly useful for time fractional equations with coefficients measurable in time, where standard semigroup methods are not readily available.
Specifically, we establish a frequency localized maximum principle for the time fractional equation
\begin{equation*}
\partial_t^\alpha u=\sum_{i,j=1}^d a_{ij}(t)D_{ij}u+f,
\qquad u(0,\cdot)=0.
\end{equation*}
Here, $\partial_t^\alpha$ is the Caputo derivative of order $\alpha\in(0,1)$, and $A(t)=(a_{ij}(t))$ is a measurable symmetric matrix satisfying $\nu I\leq A(t)\leq KI$.
Our main result states that, for a sufficiently regular solution $u$,
\begin{equation} \label{eq_intro_frequency}
\partial_t^\alpha\|\Delta_k u\|_k + \kappa 2^{2k} \|\Delta_k u\|_k \leq\|\Delta_k f\|_k, \quad k\geq1,
\end{equation}
almost everywhere in time.
Here, the constant $\kappa>0$ is independent of $k$.
See Theorem~\ref{thm_frequency} for the precise statement.

It is worth noting that the above result holds for coefficients merely measurable in time and independent of $x$.
The estimate \eqref{eq_intro_frequency} does not hold in general when the coefficients are allowed to depend also on $x$.
This failure occurs even for smooth uniformly elliptic coefficients; see Example~\ref{ex_variable_coefficients}.

As an application, we establish  partial Schauder estimates and solvability for the time fractional equation
\begin{align} \label{eq_intro}
        \partial_t^\alpha u(t,x) &= \sum_{i,j=1}^{d} a_{ij}(t,x) D_{ij} u(t,x) + \sum_{i=1}^{d} b_i(t,x) D_{i} u(t,x) \nonumber
        \\
        &\quad+ c(t,x) u(t,x) + f(t,x)
\end{align}
on $(0,T)\times\bR^d$.
The leading coefficients $a_{ij}$ are symmetric and uniformly elliptic.
For $\gamma>0$, the coefficients are assumed to belong to $L_\infty((0,T); \Lambda^\gamma(\bR^d))$, where $\Lambda^\gamma(\bR^d) = B^\gamma_{\infty,\infty}(\bR^d)$ is the H\"older-Zygmund space. 
We emphasize that the coefficients are merely strongly measurable in time.
Under Assumption~\ref{asscoeff}, for every $f\in L_p((0,T),\omega dt; \Lambda^\gamma(\bR^d))$, there is a unique solution satisfying
\begin{align} 
&\|\partial_t^\alpha u\|_{L_p((0,T),\omega dt; \Lambda^\gamma(\bR^d))} + \|u\|_{L_p((0,T),\omega dt; \Lambda^{\gamma+2}(\bR^d))}\notag\\
&\leq N\|f\|_{L_p((0,T),\omega dt; \Lambda^\gamma(\bR^d))}.\label{eq_intro_estimate}
\end{align}
Here, the estimate holds for $p\in(1,\infty)$ with $\omega\in A_p(\bR)$, and for $p=\infty$ with $\omega\equiv1$.
Since temporal regularity is considered in an $L_p$ setting, the fractional derivative is specified through the definition of the solution space $\mathbb{H}_{p,\omega,0}^{\alpha,\gamma+2}(S,T)$ in Section~\ref{sec_main_results}. 
The full solvability statement is given in Theorem~\ref{thm_main}.

The literature on parabolic equations with the ordinary time derivative, corresponding formally to $\alpha=1$, is extensive.
Since our main concern is time fractional equations, we focus on partial Schauder estimates related to \eqref{eq_intro_estimate}.
Such estimates have been established in \cite{C24,CR26,kim2018lp,krylov2002calderon,krylov2003parabolic,stinga2021holder} using kernel representations and estimates for operators with coefficients independent of the spatial variable. 
See also the references therein.

There is also an extensive literature on regularity and solvability for time fractional equations. 
One approach to studying time fractional equations is based on fundamental solutions.
Eidelman and Kochubei~\cite{EK04} used a parametrix construction for operators with spatially variable coefficients.
Kim, Kim, and Lim~\cite{KKL17} obtained an $L_q(L_p)$ theory by heat-kernel estimates and perturbation arguments under the assumption that the leading coefficients are piecewise continuous in time.
Chen, Kim, Kumagai, and Wang~\cite{CKKW20} gave representations for inhomogeneous equations associated with a fixed semigroup generator.
Kim, Park, and Ryu~\cite{KPR21} used kernel estimates to establish a mixed-norm Sobolev regularity theory for space-time nonlocal operators.
To the best of our knowledge, it is not known whether \eqref{eq_intro_estimate} can be obtained through the analysis of fundamental solutions when the coefficients are merely measurable in time.
Even for $a_{ij}=a_{ij}(t)$, the spatial Fourier transform gives a fractional differential equation with a time-dependent coefficient, to which the usual exponential formula for the heat equation does not apply.
The difficulty of obtaining a representation formula with coefficients measurable in time was also discussed in \cite{DK19}.

Several results have also been established without relying on fundamental solutions. Luchko~\cite{Luchko09} proved a maximum principle and obtained uniqueness and continuous dependence for classical solutions.
Zacher~\cite{Zacher09} established weak solvability for abstract integro-differential equations, with applications to divergence-form equations with bounded measurable coefficients.
For nondivergence-form equations, Dong and Kim~\cite{DK19} proved $L_p$ estimates with coefficients measurable in time and having small mean oscillations in space.
Weighted mixed-norm estimates were subsequently obtained in \cite{DK21, DK23}.
These results concern Sobolev spaces with finite spatial integrability exponents.
For the ordinary time derivative,  the present authors~\cite{CR26} proved solvability in weighted Sobolev-Zygmund spaces, including nonzero initial data.

However, to the best of our knowledge, partial Schauder estimates have not yet been established for time fractional equations. 
When the coefficients are independent of both time and space, such estimates can be obtained through the analysis of fundamental solutions. 
For coefficients depending on time or space, however, this motivates the development of methods that do not rely on fundamental solutions.
Nevertheless, even for the classical heat equation corresponding to $\alpha=1$, a kernel-free approach in Sobolev-Zygmund spaces does not appear to have been developed. 
To this end, we proceed as follows. 
When $a_{ij}=a_{ij}(t)$, instead of relying on fundamental solutions, we apply the comparison principle for time fractional differential inequalities (Corollary~\ref{cor_comparison}) to \eqref{eq_intro_frequency}.
Estimates for the Mittag-Leffler kernel and the weighted Hardy-Littlewood maximal inequality then yield the a priori estimate.
We extend the estimate \eqref{eq_intro_estimate} to equations with spatially dependent coefficients by applying localization arguments in both space and time.
For existence, we first construct a solution for the Laplacian using the known fractional heat kernel and then apply the method of continuity.

The paper is organized as follows.
Section~\ref{sec_main_results} introduces the function spaces and states the main results.
Section~\ref{sec_max} constructs the equivalent norms and proves the frequency localized maximum principle.
Section~\ref{sec_apriori} establishes the a priori estimate, first for coefficients independent of space and then for general coefficients.
Section~\ref{sec_solvability} proves existence and uniqueness.

We conclude this section by collecting the notation and conventions used throughout the paper.
Let $\bN=\{1,2,\ldots\}$ and $\bN_0=\bN\cup\{0\}$.
For a multi-index $\beta\in\bN_0^d$, write $|\beta|=\beta_1+\cdots+\beta_d$ and $D^\beta=D_1^{\beta_1}\cdots D_d^{\beta_d}$, where $D_i=\partial/\partial x_i$; $D^2u=(D_{ij}u)_{i,j=1}^d$ denotes the Hessian.
The spaces $\cS(\bR^d)$ and $\cS'(\bR^d)$ are the Schwartz space and its space of tempered distributions, respectively, and $C_c^\infty$ denotes the space of smooth compactly supported functions.
Our Fourier transform convention is
\begin{equation*}
\cF[h](\xi)=\int_{\bR^d}e^{-ix\cdot\xi}h(x)dx,
\qquad
\cF^{-1}[g](x)=(2\pi)^{-d}\int_{\bR^d}e^{ix\cdot\xi}g(\xi)d\xi.
\end{equation*}
We denote convolution by $\ast$, the indicator of $E$ by $1_E$, and the integral average by
\begin{equation*}
    \aint_E h(x)dx=|E|^{-1}\int_E h(x)dx.
\end{equation*}
For a Banach space $X$, $C([0,T];X)$ and $C^1([0,T];X)$ have their usual meanings, and $L_p$ spaces of $X$-valued functions are Bochner spaces.
Unspecified spatial norms are taken over $\bR^d$.
The symbol $\star$ indicates the variable on which a time operator acts.
Constants denoted by $N$ may change from line to line, with dependence indicated when needed.
We write $a\lesssim b$ for $a\leq Nb$ and $a\approx b$ for two-sided bounds.

\section{Main results}
\label{sec_main_results}

\subsection{Preliminaries}

We first introduce function spaces which will be used in this paper.
Let $\psi\in\cS(\bR^d)$ be such that
\begin{equation*}
    \cF[\psi]\geq0, \qquad \sum_{k\in\bZ}\cF[\psi](2^{-k}\xi)=1 \quad \text{ for }\xi\neq0,
\end{equation*}
and
\begin{equation*}
    \text{supp}(\cF[\psi]) \subset \left\{\xi\in\bR^d: \frac{1}{2}\leq|\xi|\leq2\right\}.
\end{equation*}
We define the standard Littlewood-Paley projections $\{\Delta_k\}_{k\geq0}$ by
\begin{equation*}
    \Delta_k h:=\psi_k*h,
\end{equation*}
where
\begin{equation*}
    \cF[\psi_k](\xi) := \cF[\psi](2^{-k}\xi), \quad k\geq1,
\end{equation*}
and
\begin{equation*}
    \cF[\psi_0](\xi) := 1-\sum_{k\geq1} \cF[\psi](2^{-k}\xi).
\end{equation*}

For $\gamma\in\bR$, we denote by $B^\gamma_{\infty,\infty}(\mathbb{R}^d)$ the Besov space equipped with the norm
\begin{equation} \label{def_hol}
    \|u\|_{B^\gamma_{\infty,\infty}(\mathbb{R}^d)} := \sup_{j\geq0}2^{j\gamma}\|\Delta_j u\|_{L_{\infty}(\mathbb{R}^d)}.
\end{equation}
Equivalently, $B^\gamma_{\infty,\infty}(\mathbb{R}^d)$ is the H\"older-Zygmund space $\Lambda^\gamma(\mathbb{R}^d)$ if $\gamma\in(0,\infty)$.

For $p \in (1,\infty)$, let $A_p(\bR)$ be the class of all 
nonnegative and locally integrable functions $\omega: \bR \to [0,\infty)$ satisfying
\begin{equation*}
    [\omega]_{A_p(\bR)} := \sup_{r>0, t\in\mathbb{R}} \left(\aint_{t-r}^{t+r} \omega(s) ds \right)\times\left(\aint_{t-r}^{t+r} \omega(s)^{-\frac{1}{p-1}} ds \right)^{p-1} < \infty.
\end{equation*}
Let $0\leq S<T<\infty$, $\gamma\in(0,\infty)$, and $\omega\in A_p(\bR)$ when $p\in(1,\infty)$, and set $\omega\equiv1$ when $p=\infty$. For $p\in(1,\infty)$, we define $\mathbf{\Lambda}_{p,\omega}^{\gamma}(S,T) :=L_p((S,T), \omega dt; \Lambda^{\gamma}(\bR^d))$ to be the set of all strongly measurable functions $u:(S,T)\to\Lambda^{\gamma}(\bR^d)$ satisfying
\begin{equation*}
    \|u\|_{\mathbf{\Lambda}_{p,\omega}^\gamma(S,T)} := \left(\int_S^T \|u(t,\cdot)\|_{\Lambda^{\gamma}(\bR^d)}^p \omega(t) dt \right)^{1/p} <\infty.
\end{equation*}
In the case when $p=\infty$, we define $\mathbf{\Lambda}_{\infty}^{\gamma}(S,T)$ to be the set of all strongly measurable functions $u:(S,T)\to\Lambda^{\gamma}(\bR^d)$ such that
\begin{equation*}
    \|u\|_{\mathbf{\Lambda}_{\infty}^\gamma(S,T)} := \text{ess sup}_{t\in (S,T)} \|u(t,\cdot)\|_{\Lambda^{\gamma}(\bR^d)} <\infty.
\end{equation*}
When $S=0$, we write $\mathbf{\Lambda}_{p,\omega}^{\gamma}(T):=\mathbf{\Lambda}_{p,\omega}^{\gamma}(0,T)$ and $\mathbf{\Lambda}_{\infty}^{\gamma}(T):=\mathbf{\Lambda}_{\infty}^{\gamma}(0,T)$.

\begin{remark}
Strong measurability is used for various techniques in this paper, such as approximation arguments. For instance, let $d=1$, $\gamma\in(0,1)$, and
\begin{equation*}
    f(t,x):=\chi(x-t)|x-t|^\gamma,\qquad (t,x)\in(0,1)\times\bR,
\end{equation*}
where $\chi\in C_c^\infty(\mathbb R)$ equals $1$ on $[-2,2]$. Then $f$ is not strongly measurable as a $\Lambda^\gamma(\bR)$-valued function. Indeed, one can verify that
\begin{equation*}
    [f(t,\cdot)-f(s,\cdot)]_{\Lambda^\gamma(\bR)} \geq 2
\end{equation*}
for $t\neq s$, so $f$ is not essentially separably valued. Here, $[\cdot]_{\Lambda^\gamma(\bR)}$ denotes the seminorm in the H\"older-Zygmund space $\Lambda^\gamma(\bR)$.

We next show that, even after mollifying $f$ in time, $f^{(\varepsilon)}$ does not converge to $f$ in the spatial H\"older norm measured in $L_p$ with respect to time.
More precisely, we extend $f$ by zero for $t\leq0$ and define its time mollification $f^{(\varepsilon)}$ by
\begin{equation*}
f^{(\varepsilon)}(t,x) := \int_0^\varepsilon \varphi_\varepsilon(s) f(t-s,x) ds, \qquad t\in(0,1), \, x\in\bR.
\end{equation*}
Here, $\varphi_\varepsilon(s) := \varepsilon^{-1}\varphi(s/\varepsilon)$ for $\varepsilon>0$, where $\varphi\in C_c^\infty((0,1))$ is a nonnegative function such that $\int_{\bR} \varphi(s) ds=1$.
Let $F_\varepsilon:=f-f^{(\varepsilon)}$. For every $t, \varepsilon \in (0,1)$,
\begin{equation*}
    \frac{F_\varepsilon(t,t+h)-F_\varepsilon(t,t)}{h^\gamma} = 1 - \int_0^{\min\{t,\varepsilon\}} \varphi_\varepsilon(r) \frac{(r+h)^\gamma-r^\gamma}{h^\gamma} dr\to1
\end{equation*}
as $h\downarrow0$. Thus, $[F_\varepsilon(t,\cdot)]_{\Lambda^\gamma(\bR)}\geq1$, and consequently,
\begin{equation*}
    \int_0^1\|f^{(\varepsilon)}(t,\cdot)-f(t,\cdot)\|_{\Lambda^\gamma(\bR)}^p dt \geq N>0,\qquad 1<p<\infty,
\end{equation*}
where $N>0$ is independent of $\varepsilon$. Thus, the claim is proved.
\end{remark}

Next, we introduce some definitions and prove some basic properties related to fractional calculus. For a strongly measurable Banach-valued function $u$ on $[S,T]$, the Riemann-Liouville fractional integral of order $\alpha\in(0,\infty)$ is defined as
\begin{equation} \label{eq8241410}
    I_S^\alpha u(t) := \frac{1}{\Gamma(\alpha)} \int_S^t (t-s)^{\alpha-1} u(s) ds,
\end{equation}
whenever the Bochner integral is finite. For $\alpha=0$, we use the convention $I_S^0u=u$.
In the following lemma, we show that $I_S^\alpha u$ is well defined.
To this end, for a locally integrable scalar function $F$ on $\bR$, we denote by $\bM_tF$ the one-dimensional centered Hardy-Littlewood maximal function
\begin{equation*}
\bM_t F(t) := \sup_{r>0} \frac{1}{2r}\int_{t-r}^{t+r}|F(s)| ds.
\end{equation*}

\begin{lemma} \label{lem_8241556}
    Let $p\in(1,\infty]$, $0\leq S<T<\infty$, $\gamma>0$, and let $\omega\in A_p(\bR)$ with $[\omega]_{A_p(\bR)}\leq K_0$ (put $\omega\equiv1$ and $K_0=1$ when $p=\infty$). Suppose that $u\in \mathbf{\Lambda}_{p,\omega}^\gamma(S,T)$. Then for any $\alpha\in(0,\infty)$, $I_S^\alpha u$ is well defined for almost every $t\in(S,T)$, and is strongly measurable with values in $\Lambda^\gamma(\bR^d)$. In particular,
    \begin{equation} \label{eq8271405}
        \|I_S^\alpha u(t)\|_{\Lambda^\gamma(\bR^d)} \leq N (t-S)^{\alpha} \bM_t \left(\|u(\star,\cdot)\|_{\Lambda^\gamma(\bR^d)} 1_{(S,T)}(\star) \right)(t), \quad t\in(S,T) \text{ (a.e.)},
    \end{equation}
    where $N=N(\alpha)$.
    Consequently, we have
    \begin{equation} \label{eq8241430}
        \|I_S^\alpha u\|_{\mathbf{\Lambda}_{p,\omega}^{\gamma}(S,T)} \leq N (T-S)^{\alpha} \|u\|_{\mathbf{\Lambda}_{p,\omega}^{\gamma}(S,T)},
    \end{equation}
            where $N=N(\alpha,p,K_0)$.
\end{lemma}

\begin{proof}
    Let $v(t):= \|u(t,\cdot)\|_{\Lambda^\gamma(\bR^d)}1_{(S,T)}(t)$. Then due to $u\in \mathbf{\Lambda}_{p,\omega}^\gamma(S,T)$, we see that $v\in L_p(\bR,\omega dt)$. Hence $\bM_t v$ is well defined on $\bR$ and finite for almost every $t\in\bR$.

    We first establish the well-definedness of $I_S^\alpha u$ and \eqref{eq8271405}.
    By \eqref{eq8241410} and
    \begin{equation*}
        \|I_S^\alpha u(t)\|_{\Lambda^\gamma(\bR^d)} \leq \frac{1}{\Gamma(\alpha)} \int_S^t (t-s)^{\alpha-1} v(s) ds = I_S^\alpha v(t),
    \end{equation*}
    it suffices to show that for $t\in(S,T)$,
    \begin{align} \label{eq8241429}
        I_S^\alpha v(t) \leq N(t-S)^{\alpha}\bM_t v(t).
    \end{align}
    Indeed,
    \begin{align*}
        I_S^\alpha v(t) &= \frac{1}{\Gamma(\alpha)} \sum_{n=0}^{\infty} \int_{t-(t-S)2^{-n}}^{t-(t-S)2^{-n-1}} (t-s)^{\alpha-1} v(s) ds
        \\
        &\leq N(\alpha) (t-S)^{\alpha-1} \sum_{n=0}^{\infty} 2^{-n(\alpha-1)} \int_{t-(t-S)2^{-n}}^{t-(t-S)2^{-n-1}} v(s) ds
        \\
        &\leq N (t-S)^{\alpha} \sum_{n=0}^{\infty} 2^{-n\alpha} \aint_{t-(t-S)2^{-n}}^{t+(t-S)2^{-n}} v(s) ds
        \\
        &\leq N(t-S)^{\alpha} \bM_tv(t) \sum_{n=0}^{\infty} 2^{-n\alpha}.
    \end{align*}
    Since $\sum_{n=0}^{\infty} 2^{-n\alpha}<\infty$, we obtain \eqref{eq8241429}.

Once the strong measurability of $I_S^\alpha u$ is proved, when $p<\infty$, \eqref{eq8241430} follows from \eqref{eq8271405} and the weighted maximal function theorem (see \cite[Theorem 2.2]{DK18}). If $p=\infty$, then we use
    \begin{align*}
        \|I_S^\alpha u(t)\|_{\Lambda^\gamma(\bR^d)} &\leq N \|u\|_{\mathbf{\Lambda}_\infty^\gamma(S,T)} \int_S^t (t-s)^{\alpha-1} ds 
        \\
        &= N (t-S)^{\alpha} \|u\|_{\mathbf{\Lambda}_\infty^\gamma(S,T)} \leq N (T-S)^{\alpha} \|u\|_{\mathbf{\Lambda}_\infty^\gamma(S,T)}.
    \end{align*}
Thus, it remains to show that $I_S^\alpha u$ is strongly measurable.
We first consider $p\in(1,\infty)$. 
    Since $u$ is strongly measurable, there is a sequence of simple functions of the form
    \begin{equation*}
        u_n(t) = \sum_{m=1}^{N_n} u_{nm}1_{A_{nm}}(t),
    \end{equation*}
    where $u_{nm}\in \Lambda^\gamma(\bR^d)$ and $A_{nm}\subset(S,T)$ are measurable, such that 
\begin{equation} \label{eq9131548}
\|u_n(t)-u(t)\|_{\Lambda^\gamma(\bR^d)} \to 0
\end{equation}
as $n\to\infty$ for almost every $t\in(S,T)$.
Let $A_n := \{t\in(S,T): \|u_n(t)-u(t)\|_{\Lambda^\gamma(\bR^d)} \leq 1\}$ and define $\tilde{u}_n(t):=u_n(t)1_{A_n}(t)$. Then $\tilde{u}_n$ is still strongly measurable, and for almost every $t\in(S,T)$,
$\tilde{u}_n(t) \to u(t)$ in $\Lambda^\gamma(\bR^d)$ as $n\to\infty$. Moreover,
\begin{equation} \label{eq9132307}
    \|\tilde{u}_n(t)\|_{\Lambda^\gamma(\bR^d)} \leq 1+\|u(t)\|_{\Lambda^\gamma(\bR^d)}
\end{equation}
and hence $\tilde{u}_n \in \mathbf{\Lambda}_{p,\omega}^{\gamma}(S,T)$.
Moreover, by the dominated convergence theorem, one can find that $\tilde{u}_n\to u$ in $\mathbf{\Lambda}_{p,\omega}^{\gamma}(S,T)$.
Since $\tilde{u}_n$ is a simple function, $I_S^\alpha\tilde{u}_n$ is also strongly measurable.
By \eqref{eq8271405} with $\tilde{u}_n-\tilde{u}_m$ instead of $u$, and the weighted maximal function theorem (see \cite[Theorem 2.2]{DK18}),
\begin{align*}
    \|I_S^\alpha\tilde{u}_n-I_S^\alpha\tilde{u}_m\|_{\mathbf{\Lambda}_{p,\omega}^{\gamma}(S,T)} \leq N(T-S)^{\alpha} \|\tilde{u}_n-\tilde{u}_m\|_{\mathbf{\Lambda}_{p,\omega}^{\gamma}(S,T)},
\end{align*}
which implies that $I_S^\alpha\tilde{u}_n$ is Cauchy in $\mathbf{\Lambda}_{p,\omega}^{\gamma}(S,T)$. 
Due to \eqref{eq9131548} and \eqref{eq9132307}, we have
\begin{equation*}
    \|I_S^\alpha \tilde{u}_n(t,\cdot) -I_S^\alpha u(t,\cdot)\|_{\Lambda^{\gamma}(\bR^d)} \leq \frac{1}{\Gamma(\alpha)} \int_S^t (t-s)^{\alpha-1} \left(1 + 2\|u(s,\cdot)\|_{\Lambda^{\gamma}(\bR^d)}\right) ds.
\end{equation*}
Thus, by the dominated convergence theorem for the Bochner integral, we have $I_S^\alpha \tilde{u}_n(t) \to I_S^\alpha u(t)$ in $\Lambda^\gamma(\bR^d)$. This yields that $I_S^\alpha u$ is strongly measurable and in $\mathbf{\Lambda}_{p,\omega}^{\gamma}(S,T)$.

Lastly, we deal with $p=\infty$. Since $T-S<\infty$, $u\in \mathbf{\Lambda}_{q,\omega_q}^\gamma(S,T)$ for any $q\in(1,\infty)$ and any $\omega_q \in A_q(\bR)$. Therefore, the strong measurability follows immediately from the result for $p<\infty$.
The lemma is proved.
\end{proof}

Below we collect some basic properties of the operator $I_S^\alpha$.

\begin{lemma} \label{lem10122003}
   Let $p\in(1,\infty]$, $0\leq S<T<\infty$, $\gamma>0$, and let $\omega\in A_p(\bR)$ with $[\omega]_{A_p(\bR)}\leq K_0$ (put $\omega\equiv1$ and $K_0=1$ when $p=\infty$). Suppose that $u\in \mathbf{\Lambda}_{p,\omega}^\gamma(S,T)$.
   
   (i) For $\alpha_1,\alpha_2\in(0,\infty)$, the semigroup identity
   \begin{equation} \label{eq8241804}
       I_S^{\alpha_1}I_S^{\alpha_2} u = I_S^{\alpha_1+\alpha_2}u
   \end{equation}
   holds for almost every $t\in (S,T)$.

   (ii) For $\alpha\in[0,\infty)$, $I_S^{1+\alpha}u$ is absolutely continuous. Moreover,
   \begin{equation} \label{eq8241829}
       \frac{d}{dt}I_S^{1+\alpha}u = I_S^\alpha u
   \end{equation}
   holds for almost every $t\in (S,T)$.

    (iii) Let $\alpha\in(0,\infty)$. If $I_S^\alpha u=0$ almost everywhere, then $u=0$ almost everywhere.
\end{lemma}

\begin{proof}
$(i)$ By Lemma \ref{lem_8241556}, both $I_S^{\alpha_1}I_S^{\alpha_2} u$ and $I_S^{\alpha_1+\alpha_2}u$ are well defined. In particular, for almost every $t\in(S,T)$,
\begin{equation*}
\int_S^t(t-r)^{\alpha_1+\alpha_2-1}\|u(r,\cdot)\|_{\Lambda^\gamma(\bR^d)} dr < \infty.
\end{equation*}
For such a $t$, by Fubini's theorem and the change of variables $s=r+(t-r)\theta$,
\begin{align*}
&\int_S^t \int_S^s (t-s)^{\alpha_1-1}(s-r)^{\alpha_2-1} \|u(r,\cdot)\|_{\Lambda^\gamma(\bR^d)} drds
\\
&\quad= \int_S^t\left(\int_r^t(t-s)^{\alpha_1-1}(s-r)^{\alpha_2-1}ds\right) \|u(r,\cdot)\|_{\Lambda^\gamma(\bR^d)} dr
\\
&\quad= B(\alpha_1,\alpha_2)\int_S^t(t-r)^{\alpha_1+\alpha_2-1} \|u(r,\cdot)\|_{\Lambda^\gamma(\bR^d)} dr < \infty,
\end{align*}
where $B(\cdot,\cdot)$ is the Beta function. Using $B(\alpha_1,\alpha_2)=\frac{\Gamma(\alpha_1)\Gamma(\alpha_2)}{\Gamma(\alpha_1+\alpha_2)}$ and Fubini's theorem (see e.g. \cite[Proposition 1.2.7]{HNVW16}),
\begin{align*}
I_S^{\alpha_1}\big( I_S^{\alpha_2}u \big)(t) &= \frac{1}{\Gamma(\alpha_1)\Gamma(\alpha_2)}
\int_S^t\int_S^s(t-s)^{\alpha_1-1}(s-r)^{\alpha_2-1} u(r) drds
\\
&=\frac{1}{\Gamma(\alpha_1)\Gamma(\alpha_2)}
\int_S^t \left(\int_r^t(t-s)^{\alpha_1-1}(s-r)^{\alpha_2-1}ds\right) u(r)dr
\\
&=\frac{1}{\Gamma(\alpha_1+\alpha_2)}
\int_S^t(t-r)^{\alpha_1+\alpha_2-1} u(r)dr
\\
&=I_S^{\alpha_1+\alpha_2} u(t).
\end{align*}

$(ii)$ Due to \eqref{eq8241804},
\begin{equation*}
\frac{d}{dt}I_S^{1+\alpha}u=\frac{d}{dt}I_S^1I_S^\alpha u.
\end{equation*}
Therefore, it suffices to consider the case $\alpha=0$. In this case, the integral is the classical one, and the result follows from \cite[Proposition 2.5.9]{HNVW16}, provided that $u\in L_1((S,T),dt;\Lambda^\gamma(\bR^d))$. This integrability follows from H\"older's inequality:
\begin{equation} \label{eq9140113}
\int_S^T\|u(t,\cdot)\|_{\Lambda^\gamma(\bR^d)} dt \leq \|u\|_{\mathbf{\Lambda}_{p,\omega}^\gamma(S,T)} \left(\int_S^T\omega(t)^{1/(1-p)} dt\right)^{(p-1)/p} < \infty
\end{equation}
if $p<\infty$, and
\begin{equation*}
\int_S^T\|u(t,\cdot)\|_{\Lambda^\gamma(\bR^d)} dt \leq (T-S) \|u\|_{\mathbf{\Lambda}_{\infty}^\gamma(S,T)} < \infty
\end{equation*}
if $p=\infty$.

$(iii)$ Let $n\in\bN$ such that $n>\alpha$. Due to \eqref{eq8241804},
\begin{equation*}
    I_S^nu= I_S^{n-\alpha} I_S^\alpha u = 0.
\end{equation*}
By \eqref{eq8241829},
\begin{equation*}
    u = \left(\frac{d}{dt}\right)^n I_S^nu = 0.
\end{equation*}
The lemma is proved.
\end{proof}

\begin{remark}
    In general, when $\alpha\in(0,1)$, $I_S^\alpha u$ is not necessarily absolutely continuous.
    Let $u(t) = (t-S)^{-\beta} v$, where $\alpha < \beta < 1$ and $v \in \Lambda^\gamma(\bR^d)$ such that $v \neq 0$. 
    For fixed $p\in(1,\infty)$, choose $\delta\in(-1,p-1)$ such that $p\beta - 1 < \delta < p-1$, and let $\omega(t) := |t-S|^\delta$. Then $\omega\in A_p(\bR)$ and $u\in \mathbf{\Lambda}_{p,\omega}^\gamma(S,T)$. Moreover,
    \begin{equation*}
        I_S^\alpha u(t) = Nv \int_S^t (t-s)^{\alpha-1} (s-S)^{-\beta} ds = Nv(t-S)^{\alpha-\beta},
    \end{equation*}
    which is singular at $t=S$.
\end{remark}

We next turn to fractional derivatives. Let $\alpha\in(0,1)$. For a sufficiently regular function $u$, for example $u\in C^1([S,T];X)$ with values in a Banach space $X$, the Riemann–Liouville fractional derivative of order $\alpha$ is defined by
\begin{equation*}
    D_{t,S}^{\alpha}u(t) :=\frac{d}{dt} I_S^{1-\alpha}u(t).
\end{equation*}
We also define the Caputo fractional derivative of order $\alpha$ by
\begin{equation*}
    \partial_{t,S}^{\alpha}u(t) = D_{t,S}^{\alpha}(u(\cdot)-u(S))(t).
\end{equation*}
When $S=0$, we simply denote $\partial_t^\alpha := \partial_{t,0}^\alpha$.

We aim to extend this definition to Banach-valued $L_p$ spaces with zero initial data.
Roughly speaking, we seek a solution $u$ in the space characterized by $u\in\mathbf{\Lambda}_{p,\omega}^{\gamma+2}(S,T)$ and $\partial_{t,S}^\alpha u\in\mathbf{\Lambda}_{p,\omega}^\gamma(S,T)$. 
However, this characterization, in general, does not uniquely determine the extension of $\partial_{t,S}^\alpha$.
Indeed, consider the case when $\alpha p<1$ and $\omega\equiv1$. For a nonzero $\phi\in C_c^\infty(\bR^d)$, let
\begin{equation*}
u_n(t,x)=t^{1/n}\phi(x),\qquad v_n(t,x)=\phi(x).
\end{equation*}
Both sequences converge to $\phi$ in $\mathbf{\Lambda}_{p,1}^{\gamma+2}(T)$, whereas $u_n(0,\cdot)=0$ and $v_n(0,\cdot)=\phi$.
Moreover,
\begin{align*}
\partial_t^\alpha u_n(t,x) &= \frac{\Gamma(1+1/n)}{\Gamma(1+1/n-\alpha)}t^{1/n-\alpha}\phi(x) \to \frac{t^{-\alpha}}{\Gamma(1-\alpha)}\phi(x),
\\
\partial_t^\alpha v_n(t,x) &= 0,
\end{align*}
where the convergence takes place in $\mathbf{\Lambda}_{p,1}^{\gamma}(T)$.

To obtain a well-defined notion of the solution space, we recall that $\partial_{t,S}^{\alpha}u$ is expressed as the time derivative of $I_S^{1-\alpha}(u-u(S))$.
Hence, we need to consider $u\in\mathbf{\Lambda}_{p,\omega}^{\gamma+2}(S,T)$ such that $I_S^{1-\alpha}u$ has a weak time derivative in $\mathbf{\Lambda}_{p,\omega}^{\gamma}(S,T)$. By \eqref{eq8241430} and
\begin{equation*}
    \mathbf{\Lambda}_{p,\omega}^{\gamma}(S,T) \subset L_1((S,T);\Lambda^\gamma(\bR^d)),
\end{equation*}
$I_S^{1-\alpha}u$ and its weak time derivative belong to $L_1((S,T);\Lambda^\gamma(\bR^d))$. 
Consequently, $I_S^{1-\alpha}u$ has an absolutely continuous representative on $[S,T]$, and its initial trace is well defined.
As for the zero initial condition, for a sufficiently regular function $v$ such that $v(S)=0$,
\begin{equation*}
I_S^{1-\alpha}v(t)=\int_S^t\partial_{t,S}^\alpha v(s) ds,
\end{equation*}
which yields $\left(I_S^{1-\alpha}v\right)(S)=0$. 

Motivated by these observations, we introduce the following function space.
We define $\bH_{p,\omega,0}^{\alpha,\gamma+2}(S,T)$ to be the space of all $u \in \mathbf{\Lambda}_{p,\omega}^{\gamma+2}(S,T)$ such that $I_S^{1-\alpha}u$ is a $\Lambda^{\gamma}(\bR^d)$-valued absolutely continuous function, $I_S^{1-\alpha}u(S,\cdot)=0$, and $\partial_{t,S}^\alpha u := \frac{d}{dt}I_S^{1-\alpha}u \in \mathbf{\Lambda}_{p,\omega}^{\gamma}(S,T)$.
Here, the norm in $\bH_{p,\omega,0}^{\alpha,\gamma+2}(S,T)$ is given by
\begin{equation*}
\|u\|_{\bH_{p,\omega,0}^{\alpha,\gamma+2}(S,T)} := \|u\|_{\mathbf{\Lambda}_{p,\omega}^{\gamma+2}(S,T)} + \|\partial_{t,S}^\alpha u\|_{\mathbf{\Lambda}_{p,\omega}^\gamma(S,T)}.
\end{equation*}
As above, we write $\bH_{p,\omega,0}^{\alpha,\gamma+2}(T):=\bH_{p,\omega,0}^{\alpha,\gamma+2}(0,T)$ and $\bH_{\infty,0}^{\alpha,\gamma+2}(S,T):=\bH_{\infty,1,0}^{\alpha,\gamma+2}(S,T)$.
 We seek strong solutions of \eqref{eq_intro} in $\bH_{p,\omega,0}^{\alpha,\gamma+2}(T)$, meaning that \eqref{eq_intro} holds as an identity in $\Lambda^\gamma(\bR^d)$ for almost every $t\in(0,T)$.

 Here, we can write
 \begin{equation*}
I_S^{1-\alpha}u(t)=\int_S^t \partial_{t,S}^{\alpha}u(s) ds \quad \text{ for almost every } t\in(S,T),
\end{equation*}
which together with Lemma \ref{lem10122003} $(iii)$ yields that the above definition is well defined.

\begin{proposition}
    Let $p\in(1,\infty]$, $0\leq S<T<\infty$, $\alpha\in(0,1)$, $\gamma>0$, and let $\omega\in A_p(\bR)$ with $[\omega]_{A_p(\bR)}\leq K_0$ (put $\omega\equiv1$ and $K_0=1$ when $p=\infty$). 
Then the space $\bH_{p,\omega,0}^{\alpha,\gamma+2}(S,T)$ is a Banach space.
\end{proposition}

\begin{proof}
    Let $u_n$ be Cauchy in $\bH_{p,\omega,0}^{\alpha,\gamma+2}(S,T)$. Then $u_n$ and $\partial_{t,S}^{\alpha} u_n$ converge to $u$ and $g$ in the corresponding spaces.
    By \eqref{eq8241430}, one can see that $I_S^{1-\alpha}u_n$ converges to $I_S^{1-\alpha}u$ in $\mathbf{\Lambda}_{p,\omega}^{\gamma}(S,T)$.
    We claim that $\partial_{t,S}^{\alpha}u=g$.
    By \eqref{eq8241804} and \eqref{eq8241829},
\begin{equation} \label{eq9290128}
    I_{S}^{\alpha}\partial_{t,S}^{\alpha} u_n = I_{S}^{\alpha} \frac{d}{dt} I_{S}^{1-\alpha} u_n = \frac{d}{dt}I_{S}^{1+\alpha} \frac{d}{dt}I_{S}^{1-\alpha} u_n = \frac{d}{dt}I_{S}^1 u_n = u_n,
\end{equation}
    which yields $u=I_S^{\alpha}g$.
    Here, for the third equality, we used $I_S^{1-\alpha}u_n(S,\cdot)=0$, which leads to
    \begin{equation*}
        I_S^1 \frac{d}{dt}I_S^{1-\alpha} u_n = I_S^{1-\alpha} u_n.
    \end{equation*}
    By \eqref{eq8241804}, $I_S^{1-\alpha}u=I_S^1g$. Thus, \eqref{eq8241829} leads to $\partial_{t,S}^{\alpha}u=g$.
    It remains to note that $I_S^{1-\alpha}u$ is absolutely continuous, which directly follows from $I_S^{1-\alpha}u=I_S^1g$. Moreover, it is obvious that $I_S^{1-\alpha}u(S,\cdot) = I_S^1g(S,\cdot) =0$.
    The proposition is proved.
\end{proof}

Next, we establish a time-mollification result in $\bH_{p,\omega,0}^{\alpha,\gamma+2}(S,T)$, together with a density result for $p\in(1,\infty)$.

\begin{lemma} \label{lem_dense}
    Let $p\in(1,\infty]$, $0\leq S<T<\infty$, $\alpha\in(0,1)$, $\gamma>0$, and let $\omega\in A_p(\bR)$ with $[\omega]_{A_p(\bR)}\leq K_0$ (put $\omega\equiv1$ and $K_0=1$ when $p=\infty$). Let $\varphi\in C_c^\infty((0,1))$ be a nonnegative function such that
\begin{equation*}
\int_{\bR} \varphi(s) ds=1,
\end{equation*}
and denote $\varphi_\varepsilon(s) := \varepsilon^{-1}\varphi(s/\varepsilon)$ for $\varepsilon>0$.
For $u\in\bH_{p,\omega,0}^{\alpha,\gamma+2}(S,T)$, extend $u$ by zero for $t\leq S$, and define its time mollification by
\begin{equation} \label{eq8252326}
u^{(\varepsilon)}(t) := \int_0^\varepsilon\varphi_\varepsilon(s)u(t-s) ds, \qquad t\in[S,T].
\end{equation}
Then $u^{(\varepsilon)}(S,\cdot)=0$, and $u^{(\varepsilon)}$ is infinitely differentiable in $t$ on $(S,T)$ with values in $\Lambda^{\gamma+2}(\bR^d)$. Moreover, we have $u^{(\varepsilon)} \in \bH_{p,\omega,0}^{\alpha,\gamma+2}(S,T)$ and
\begin{equation} \label{eq8251612}
\sup_{\varepsilon>0} \|u^{(\varepsilon)}\|_{\bH_{p,\omega,0}^{\alpha,\gamma+2}(S,T)} \leq N\|u\|_{\bH_{p,\omega,0}^{\alpha,\gamma+2}(S,T)},
\end{equation}
where $N=N(p,K_0,\varphi)$.
In particular, if $p\in(1,\infty)$, then 
\begin{equation} \label{eq8251608}
\|u^{(\varepsilon)}-u\|_{\bH_{p,\omega,0}^{\alpha,\gamma+2}(S,T)} \to 0
\end{equation}
as $\varepsilon \to 0$.
\end{lemma}

\begin{proof}
We first note that the zero extension of $u$ is strongly measurable on $(-\infty,T)$ with values in $\Lambda^{\gamma+2}(\bR^d)$. Thus, $u^{(\varepsilon)}$ is well defined. Moreover, it follows immediately from the definition of $u^{(\varepsilon)}$ that $u^{(\varepsilon)}(S,\cdot)=0$.
By the standard mollification argument, the infinite differentiability of $u^{(\varepsilon)}$ follows in the same way as in the real-valued case.

Next, we show that $u^{(\varepsilon)} \in \bH_{p,\omega,0}^{\alpha,\gamma+2}(S,T)$ and \eqref{eq8251612}. Since $u^{(\varepsilon)}$ is infinitely differentiable in $t$, due to \eqref{eq8241804}, one can deduce that $I_S^{1-\alpha}u^{(\varepsilon)}= I_S^{2-\alpha}\frac{d}{dt}u^{(\varepsilon)}$. Thus, $I_S^{1-\alpha}u^{(\varepsilon)}$ is absolutely continuous.
Since $\varphi\in C_c^\infty((0,1))$, there exists $\delta=\delta(\varepsilon)>0$ such that $u^{(\varepsilon)}(t)=0$ for  $t\in(S,S+\delta)$. Thus, $I_S^{1-\alpha}u^{(\varepsilon)}(S)=0$.
To conclude $u^{(\varepsilon)} \in \bH_{p,\omega,0}^{\alpha,\gamma+2}(S,T)$, it remains to show that $u^{(\varepsilon)} \in \mathbf{\Lambda}_{p,\omega}^{\gamma+2}(S,T)$ and $\partial_{t,S}^{\alpha}u^{(\varepsilon)} \in \mathbf{\Lambda}_{p,\omega}^{\gamma}(S,T)$.
By Fubini's theorem, one can prove that $I_S^{1-\alpha}u^{(\varepsilon)}(t) = (I_S^{1-\alpha}u)^{(\varepsilon)}(t)$,
which yields
\begin{equation} \label{eq8251609}
    \partial_{t,S}^{\alpha}u^{(\varepsilon)} = (\partial_{t,S}^{\alpha}u)^{(\varepsilon)}.
\end{equation}
When $p=\infty$, the claims follow immediately.

To conclude the claims when $p\in(1,\infty)$, we note that for $v\in \mathbf{\Lambda}_{p,\omega}^{\gamma'}(S,T)$ with $\gamma'>0$, by a norm inequality for the Bochner integral,
\begin{align} \label{eq9140120}
    \|v^{(\varepsilon)}(t,\cdot)\|_{\Lambda^{\gamma'}} &\leq \int_0^\varepsilon \varphi_{\varepsilon}(s) \|v(t-s,\cdot)\|_{\Lambda^{\gamma'}} ds \nonumber
    \\
    &\leq N \bM_t\left(  \|v(\star,\cdot)\|_{\Lambda^{\gamma'}} 1_{(S,T)}(\star) \right) (t).
\end{align}
Then the claim follows from the weighted maximal function theorem.

Lastly, we prove \eqref{eq8251608}.
By \eqref{eq9140113}, if $v \in \mathbf{\Lambda}_{p,\omega}^{\gamma'}(S,T)$ with $\gamma'>0$, then we have $v\in L_1((S,T), dt; \Lambda^{\gamma'}(\bR^d))$. Thus, by the standard approximation result for mollifications (see, e.g., \cite[Proposition 1.2.32]{HNVW16}), we see that $v^{(\varepsilon)}(t,\cdot)\to v(t,\cdot)$ in $\Lambda^{\gamma'}(\bR^d)$ for almost every $t\in (S,T)$.
By \eqref{eq9140120} and the dominated convergence theorem, we actually have $v^{(\varepsilon)}\to v$ in $\mathbf{\Lambda}_{p,\omega}^{\gamma'}(S,T)$ for $p\in(1,\infty)$.
Therefore, \eqref{eq8251609} yields the desired result.
The lemma is proved.
\end{proof}

\subsection{Statements of the main results} \label{sec_statement}

We first formulate a frequency localized maximum principle for sufficiently regular solutions $u\in C^1([0,T];\Lambda^{\gamma+2}(\bR^d))$. To this end, we introduce a space $X_k$ ($k\in\bN_0$) whose norm is equivalent to the $L_\infty(\bR^d)$-norm on the range of the Littlewood-Paley projection $\Delta_k$.

For $k\in\bN_0$, let
\begin{equation*}
    \cA_0:=\{\xi\in\bR^d:|\xi|\leq2\}, \qquad \cA_k:=\{\xi\in\bR^d:2^{k-1}\leq|\xi|\leq2^{k+1}\}, \quad k\geq1,
\end{equation*}
and define
\begin{equation*}
    X_k := \left\{ h\in L_\infty(\bR^d): \text{supp}(\cF[h]) \subset \cA_k \right\},
\end{equation*}
where the Fourier transform is understood in the sense of
tempered distributions. Then the space $X_k$ is a closed real linear subspace of $L_\infty(\bR^d)$. Indeed, let $h_n\in X_k$ and suppose that $h_n\to h$ in $L_\infty(\bR^d)$. Then for every $\varphi\in\cS(\bR^d)$,
\begin{equation*}
    |\langle h_n-h,\varphi\rangle| \leq \|h_n-h\|_{L_\infty(\bR^d)} \|\varphi\|_{L_1(\bR^d)} \to 0.
\end{equation*}
Thus, $h_n\to h$ in $\cS'(\bR^d)$, and hence $\cF[h_n]\to\cF[h]$ in $\cS'(\bR^d)$. If $\phi\in C_c^\infty(\bR^d\setminus\cA_k)$, then
\begin{equation*}
    \langle \cF[h],\phi \rangle = \lim_{n\to\infty} \langle \cF[h_n],\phi \rangle = 0.
\end{equation*}
Thus, $\text{supp}(\cF[h])\subset\cA_k$, and $h\in X_k$. The claim is proved. 

For $0 < \nu \leq K$, we denote
\begin{equation*}
    \bA_{\nu,K} := \left\{ B\in\bR^{d\times d}: B=B^T, \, \nu I\leq B\leq KI \right\}.
\end{equation*}
Here, $B^T$ denotes the transpose of $B$.
For $B=(b_{ij}) \in \bA_{\nu,K}$, we write $L_B:=\sum_{i,j=1}^d b_{ij}D_{ij}$. Then the associated heat semigroup $e^{rL_B}$ is given by
\begin{equation} \label{eq8162343}
    e^{rL_B}h(x) = p_B(r,\cdot)*h(x), \quad r>0,
\end{equation}
where
\begin{equation*}
    p_B(r,x) := \cF^{-1}[e^{-r\sum_{i,j=1}^d b_{ij}\xi_i\xi_j}](x).
\end{equation*}

The proof of the following proposition is given in Section \ref{sec_proof_frequency}.

\begin{proposition} \label{prop_knorm}
    Let $0<\nu\leq K$. 
    For $k\geq1$ and $\kappa>0$, we define
    \begin{equation} \label{knorm}
\|h\|_k := \sup_{\substack{B\in\bA_{\nu,K} \\ r>0}} e^{2^{2k}\kappa r}\|e^{rL_B}h\|_{L_\infty(\bR^d)}, \qquad h\in X_k.
\end{equation}
Then there exists a constant $c_0=c_0(d,\nu,K,\psi)>0$ such that for every $\kappa\in(0,c_0)$, \eqref{knorm} defines a norm on $X_k$ satisfying
\begin{equation}\label{eq8170137}
    \|h\|_{L_\infty(\bR^d)}\leq\|h\|_k\leq N\|h\|_{L_\infty(\bR^d)},
\end{equation}
where $N=N(d,\nu,K,\psi,\kappa)$.
Consequently, $(X_k,\|\cdot\|_k)$ is a Banach space for every $k\geq1$.
\end{proposition}

In the above proposition, we emphasize that the norm \eqref{knorm} depends on the choice of $\kappa$.

We now state our first result. 

\begin{theorem}[A frequency localized maximum principle] \label{thm_frequency}
 Let $T\in(0,\infty)$, $\alpha\in(0,1)$, and $\gamma\in(0,\infty)$. Suppose that the coefficients $a_{ij}=a_{ij}(t)$ are measurable, independent of $x$, and satisfy $a_{ij}(t)=a_{ji}(t)$ and
\begin{equation} \label{eq9281844}
\nu|\xi|^2 \leq \sum_{i,j=1}^d a_{ij}(t)\xi_i\xi_j \leq K|\xi|^2,\qquad \xi\in\bR^d.
\end{equation}
Let $c_0=c_0(d,\nu,K,\psi)>0$ be the constant given in Proposition \ref{prop_knorm}, and fix $\kappa\in(0,c_0)$. For $k\geq1$, we equip $X_k$ with the norm \eqref{knorm} corresponding to $\kappa$, while $X_0$ is equipped with the $L_\infty(\bR^d)$-norm.
 Assume that $u\in C^1([0,T];\Lambda^{\gamma+2}(\bR^d))$ satisfies
\begin{equation*}
\partial_t^\alpha u(t,\cdot)=\sum_{i,j=1}^d a_{ij}(t)D_{ij}u(t,\cdot)+f(t,\cdot), \qquad u(0,\cdot)=0,
\end{equation*}
in $\Lambda^\gamma(\bR^d)$ for almost every $t\in(0,T)$, where $f\in L_\infty((0,T);\Lambda^\gamma(\bR^d))$.
Then for every $k\in\bN_0$, we have $\Delta_k u\in C^1([0,T];X_k)$ and $\Delta_k f\in L_\infty((0,T);X_k)$.
Moreover, for every $k\geq1$,
\begin{equation} \label{eq_frequency}
\partial_t^\alpha \|\Delta_k u(\star,\cdot)\|_k(t) + 2^{2k}\kappa\|\Delta_k u(t,\cdot)\|_k \leq \|\Delta_k f(t,\cdot)\|_k,
\end{equation}
while for $k=0$,
\begin{equation} \label{eq_frequency_low}
\partial_t^\alpha \|\Delta_0 u(\star,\cdot)\|_0(t) \leq \|\Delta_0 f(t,\cdot)\|_0.
\end{equation}
Both inequalities hold for almost every $t\in(0,T)$.
Here, for $k\geq1$, the norm on $X_k$ is the one defined in \eqref{knorm} with $\nu$ and $K$ given by \eqref{eq9281844}.
\end{theorem}

We now present the assumptions on the coefficients.

\begin{assumption}[$\gamma$] \label{asscoeff}
There exist positive constants $\nu$ and $K$ satisfying the following conditions:
\begin{enumerate}
\item For all $(t,x)\in(0,\infty)\times\bR^d$, $a_{ij}=a_{ji}$ and
\begin{equation*}
\nu|\xi|^2\leq\sum_{i,j=1}^da_{ij}(t,x)\xi_i\xi_j\leq K|\xi|^2,\qquad \xi\in\bR^d.
\end{equation*}

\item The coefficients $a_{ij}$, $b_i$, and $c$ are $\Lambda^\gamma(\bR^d)$-valued strongly measurable functions such that
\begin{equation*}
\sum_{i,j=1}^d\|a_{ij}\|_{L_\infty((0,\infty);\Lambda^\gamma(\bR^d))}+\sum_{i=1}^d\|b_i\|_{L_\infty((0,\infty);\Lambda^\gamma(\bR^d))}+\|c\|_{L_\infty((0,\infty);\Lambda^\gamma(\bR^d))}\leq K.
\end{equation*}
\end{enumerate}
\end{assumption}

Finally, we state our solvability result.

\begin{theorem} \label{thm_main}
Let $T\in(0,\infty)$, $\alpha\in(0,1)$, $\gamma\in(0,\infty)$, and $p\in(1,\infty]$. Let $\omega\in A_p(\bR)$ with $[\omega]_{A_p(\bR)}\leq K_0$ (put $\omega\equiv1$ and $K_0=1$ when $p=\infty$).
Suppose that Assumption \ref{asscoeff} holds. Then for any $f \in \mathbf{\Lambda}_{p,\omega}^\gamma(T)$, there exists a unique solution $u \in \bH_{p,\omega,0}^{\alpha,\gamma+2}(T)$ to \eqref{eq_intro}.
Moreover,
\begin{equation} \label{eq_main_est}
\|\partial_t^\alpha u\|_{\mathbf{\Lambda}_{p,\omega}^\gamma(T)} + \|u\|_{\mathbf{\Lambda}_{p,\omega}^{\gamma+2}(T)} \leq  N \|f\|_{\mathbf{\Lambda}_{p,\omega}^\gamma(T)},
\end{equation}
where $N$ depends only on $d$, $\alpha$, $\gamma$, $p$, $K_0$, $\nu$, $K$, and $T$.
\end{theorem}

\section{A frequency localized maximum principle} \label{sec_max}

In this section, we prove Proposition \ref{prop_knorm} and the frequency localized maximum principle in Theorem \ref{thm_frequency}. We first collect some auxiliary results.

\subsection{Auxiliary results}

\begin{lemma} \label{lem8170201}
Let $(X,\|\cdot\|_X)$ be a real Banach space and let $0\neq h\in X$. Then there exists a continuous linear functional $l_h\in X^*$ such that
\begin{equation} \label{eq8170201}
    \|l_h\|_{X^*}=1, \qquad l_h(h)=\|h\|_X.
\end{equation}
\end{lemma}

\begin{proof}
This is in fact the Hahn-Banach theorem.
Let $Y_h := \{ch: c\in\bR\} \subset X$, and define $l_{h,0}:Y_h\to\bR$ by
\begin{equation*}
    l_{h,0}(ch):=c\|h\|_X, \quad c\in\bR.
\end{equation*}
Since $h\neq0$, the functional $l_{h,0}$ is well-defined and linear. 
Moreover, 
$$
\|l_{h,0}\|_{(Y_h,\|\cdot\|_X)^*}=1,\quad \text{and}\quad l_{h,0}(h)=\|h\|_X.
$$ 
By the  Hahn-Banach theorem, $l_{h,0}$ extends to a continuous linear functional $l_h$ on $X$ satisfying $\|l_h\|_{X^*} \leq 1$.
Since $l_h(h)=\|h\|_X$, we also have $\|h\|_X=|l_h(h)| \leq \|l_h\|_{X^*} \|h\|_X$, which yields $\|l_h\|_{X^*}=1$. The lemma is proved.
\end{proof}

\begin{lemma} \label{lem8170203}
Let $(X,\|\cdot\|_X)$ be a real Banach space and $\alpha\in(0,1)$. For $u\in C^1([0,T];X)$, we denote
\begin{equation*}
    v(t) := \|u(t)\|_X.
\end{equation*}
Fix $t\in(0,T]$. If $v(t)>0$, let $l_t\in X^*$ satisfy
\begin{equation*}
    \|l_t\|_{X^*}=1, \qquad l_t(u(t))=v(t).
\end{equation*}
Then we have
\begin{equation}\label{eq8170206}
    \partial_t^\alpha v(t)\leq l_t(\partial_t^\alpha u(t)).
\end{equation}
If $v(t)=0$, then \eqref{eq8170206} holds with $l_t=0$.
\end{lemma}

\begin{proof}
We first notice that $v$ is Lipschitz continuous on $[0,T]$, and the existence of $l_t$ follows from Lemma \ref{lem8170201}. For every $s\in[0,t)$,
\begin{align*}
    v(t)-v(s)&=l_t(u(t))-\|u(s)\|_X
    \\
    &\leq l_t(u(t))-l_t(u(s))=l_t(u(t)-u(s)).
\end{align*}

We claim that
\begin{align} \label{eqMarch}
    \partial_t^\alpha v(t) =\frac{1}{\Gamma(1-\alpha)} \left[ \frac{v(t)-v(0)}{t^\alpha} +\alpha\int_0^t \frac{v(t)-v(s)}{(t-s)^{1+\alpha}} ds \right]
\end{align}
(see also \cite[Remark 3.7]{KW25}).
Using
\begin{equation*}
    v(s)-v(0)=\int_0^s v'(r) dr
\end{equation*}
and Fubini's theorem,
\begin{align*}
I_0^{1-\alpha}(v-v(0))(t) &= \frac{1}{\Gamma(1-\alpha)} \int_0^t (t-s)^{-\alpha} \int_0^s v'(r) drds
\\
&= \frac{1}{\Gamma(2-\alpha)} \int_0^t (t-r)^{1-\alpha}v'(r) dr,
\end{align*}
and hence
\begin{equation} \label{eq9142031}
    \partial_t^\alpha v(t) = \frac{1}{\Gamma(1-\alpha)} \int_0^t (t-r)^{-\alpha}v'(r) dr.
\end{equation}
On the other hand, by $v(t)-v(s)=\int_s^t v'(r) dr$ and Fubini's theorem,
\begin{align*}
    \alpha\int_0^t
\frac{v(t)-v(s)}{(t-s)^{1+\alpha}} ds &= \int_0^t \left((t-r)^{-\alpha}-t^{-\alpha}\right)v'(r) dr
\\
&= \int_0^t(t-r)^{-\alpha}v'(r) dr - \frac{v(t)-v(0)}{t^\alpha}.
\end{align*}
This together with \eqref{eq9142031} proves the claim.

Due to \eqref{eqMarch}, we obtain
\begin{align*}
    \partial_t^\alpha v(t) \leq \frac{1}{\Gamma(1-\alpha)}\left[\frac{l_t(u(t)-u(0))}{t^\alpha}+\alpha\int_0^t\frac{l_t(u(t)-u(s))}{(t-s)^{1+\alpha}}ds\right] =l_t(\partial_t^\alpha u(t)).
\end{align*}
If $v(t)=0$, then $v(t)-v(s)=-v(s)\leq0$ for every $s<t$. Thus, again by \eqref{eqMarch}, we have $\partial_t^\alpha v(t)\leq0$, and thus \eqref{eq8170206} follows from $l_t=0$.
The lemma is proved.
\end{proof}

Recall that
\begin{equation*}
    \cA_0:=\{\xi\in\bR^d:|\xi|\leq2\}, \qquad \cA_k:=\{\xi\in\bR^d:2^{k-1}\leq|\xi|\leq2^{k+1}\}, \quad k\geq1,
\end{equation*}
\begin{equation*}
    X_k := \left\{ h\in L_\infty(\bR^d): \text{supp}(\cF[h]) \subset \cA_k \right\},
\end{equation*}
\begin{equation*}
    \bA_{\nu,K} := \left\{ B\in\bR^{d\times d}: B=B^T, \quad \nu I\leq B\leq KI \right\}, \quad 0 < \nu \leq K,
\end{equation*}
and $e^{rL_B}$ is the semigroup generated by $L_B$, where $B=(b_{ij}) \in \bA_{\nu,K}$.

\begin{lemma} \label{lem8161941}
Let $k\geq1$, $B\in\bA_{\nu,K}$, $r>0$, and $h\in L_\infty(\bR^d)$ satisfying $\text{supp}(\cF[h])\subset\cA_k$. Then there exist constants $N\geq1$ and $c_0>0$, depending only on $d$, $\nu$, $K$, and $\psi$, such that
\begin{equation*}
    \|e^{rL_B}h\|_{L_\infty(\bR^d)} \leq N e^{-2^{2k}c_0r} \|h\|_{L_\infty(\bR^d)}.
\end{equation*}
\end{lemma}

\begin{proof}
Let $\chi\in C_c^\infty(\bR^d)$ such that $\text{supp}(\chi)\subset \{1/4\leq|\xi|\leq 4\}$ and $\chi=1$ if $1/2\leq|\xi|\leq2$, and set $\chi_k(\xi):=\chi(2^{-k}\xi)$. Then by \eqref{eq8162343},
\begin{align*}
    \|e^{rL_B}h\|_{L_\infty(\bR^d)} &= \|(\cF^{-1}[\chi_k])*(e^{rL_B}h)\|_{L_\infty(\bR^d)}
    \\
    &\leq \|(\cF^{-1}[\chi_k])*p_B(r,\cdot)\|_{L_1(\bR^d)} \|h\|_{L_\infty(\bR^d)}.
\end{align*}
By the change of variables $\xi\to2^k\xi$,
\begin{equation*}
    \|(\cF^{-1}[\chi_k])*p_B(r,\cdot)\|_{L_1(\bR^d)} \leq \left\|\cF^{-1}\left[e^{-r2^{2k} \sum_{i,j=1}^d b_{ij}\xi_i\xi_j}\chi(\xi)\right](\cdot) \right\|_{L_1(\bR^d)}.
\end{equation*}
Thus, it suffices to show that
\begin{equation}\label{eq8170023}
    \left\|K_{\tau}\right\|_{L_1(\bR^d)}\leq Ne^{-c_0\tau},
\end{equation}
where $\tau>0$ and
\begin{equation*}
    K_{\tau}(x) := \cF^{-1}\left[e^{-\tau\sum_{i,j=1}^d b_{ij}\xi_i\xi_j}\chi(\xi)\right](x).
\end{equation*}
Let $m>d/2$ be an integer. By H\"older's inequality and the Plancherel theorem, we have
\begin{align} \label{eq8170031}
    \|K_{\tau}\|_{L_1(\bR^d)} &\leq \left\| (1+|x|^2)^{-m/2} \right\|_{L_2(\bR^d)} \left\| (1+|x|^2)^{m/2}K_{\tau} \right\|_{L_2(\bR^d)} \nonumber
\\
    &\leq N\sum_{|\beta|\leq m}\left\|x^{\beta} K_{\tau}\right\|_{L_2(\bR^d)} \nonumber
\\
    &\leq N\sum_{|\beta|\leq m}
    \left\|D_\xi^\beta\left[e^{-\tau\sum_{i,j=1}^d b_{ij}\xi_i\xi_j}\chi(\xi)\right]\right\|_{L_2(\bR^d)} \nonumber
    \\
    &\leq N (1+\tau)^me^{-c\nu\tau} \sum_{|\beta|\leq m}\|D_\xi^\beta\chi\|_{L_2(\bR^d)} \leq Ne^{-c_0\tau}.
\end{align}
Here, for the last inequality, we used $\sum_{i,j=1}^d b_{ij}\xi_i\xi_j\geq \nu/16$ on $\text{supp}(\chi)$. 
Thus, \eqref{eq8170023} is obtained. The lemma is proved.
\end{proof}

\begin{lemma}
Let $k\geq1$ and assume that $X_k$ is equipped with the norm defined in \eqref{knorm}. Then for every $B\in\bA_{\nu,K}$, $r>0$, and $h\in X_k$,
\begin{equation}
\label{eq8170133}
\|e^{rL_B}h\|_k \leq e^{-2^{2k}\kappa r}\|h\|_k.
\end{equation}
\end{lemma}

\begin{proof}
For $A,B\in\bA_{\nu,K}$ and $r,s>0$,
\begin{equation*}
e^{rL_B}e^{sL_A}=e^{(r+s)L_C}, \qquad C:=\frac{rB+sA}{r+s}.
\end{equation*}
Here, we see that $C\in\bA_{\nu,K}$, which yields
\begin{align*}
\|e^{sL_A}h\|_k &= \sup_{\substack{B\in\bA_{\nu,K} \\ r>0}}e^{2^{2k} \kappa r} \|e^{rL_B}e^{sL_A}h\|_{L_\infty(\bR^d)} \\
&\leq e^{-2^{2k}\kappa s}\sup_{\substack{C\in\bA_{\nu,K} \\ \tau>s}}e^{2^{2k}\kappa\tau} \|e^{\tau L_C}h\|_{L_\infty(\bR^d)} \leq e^{-2^{2k}\kappa s} \|h\|_k.
\end{align*}
The lemma is proved.
\end{proof}

\begin{lemma} \label{lem8170202}
Let $B\in\bA_{\nu,K}$, $k\geq1$, and $h \in X_k$. If $l_h$ satisfies \eqref{eq8170201} with $X_k$ in place of $X$, then
\begin{equation} \label{eq8170202}
    l_h(L_Bh)\leq-2^{2k}\kappa\|h\|_k.
\end{equation}
\end{lemma}

\begin{proof}
Note that the case $h=0$ is obvious. Thus, we may assume that $h\neq 0$.
By \eqref{eq8170133}, for every $s>0$,
\begin{equation*}
    l_h(e^{sL_B}h)\leq|l_h(e^{sL_B}h)|\leq\|e^{sL_B}h\|_k\leq e^{-2^{2k}\kappa s}\|h\|_k.
\end{equation*}
Since $l_h$ is linear and $l_h(h)=\|h\|_k$,
\begin{equation} \label{eq8170203}
    l_h\left(\frac{e^{sL_B}h-h}{s}\right)\leq\frac{e^{-2^{2k}\kappa s}-1}{s}\|h\|_k.
\end{equation}
We claim that
\begin{equation}\label{eq8170204}
    \frac{e^{sL_B}h-h}{s}\to L_Bh \quad\text{in }(X_k,\|\cdot\|_k)
\end{equation}
as $s\to0$. Let $\chi_k\in C_c^\infty(\bR^d)$ such that $\chi_k=1$ on $\cA_k$. Since $\cF[h] = \chi_k\cF[h]$,
\begin{equation*}
    \frac{e^{sL_B}h-h}{s}-L_Bh=\cF^{-1}[m_s]*h,
\end{equation*}
where
\begin{equation*}
    m_s(\xi):=\left(\frac{e^{-s\sum_{i,j=1}^d b_{ij}\xi_i\xi_j}-1}{s} + \sum_{i,j=1}^d b_{ij}\xi_i\xi_j \right)\chi_k(\xi).
\end{equation*}
For every multi-index $\beta$, $D_\xi^\beta m_s$ converges uniformly to zero on $\text{supp}(\chi_k)$ as $s\to0$. Thus, as in \eqref{eq8170031},
\begin{equation*}
    \|\cF^{-1}[m_s]\|_{L_1(\bR^d)} \to 0 \quad \text{ as } s\to 0,
\end{equation*}
which yields
\begin{equation*}
    \left\|\frac{e^{sL_B}h-h}{s}-L_Bh\right\|_{L_\infty(\bR^d)} \to 0 \quad \text{ as } s\to0.
\end{equation*}
This together with \eqref{eq8170137} proves \eqref{eq8170204}.
Thus, letting $s\to0$ in \eqref{eq8170203}, we obtain \eqref{eq8170202}.
The lemma is proved.
\end{proof}

\subsection{Proofs of Proposition \ref{prop_knorm} and Theorem \ref{thm_frequency}} \label{sec_proof_frequency}

We first present the proof of Proposition \ref{prop_knorm}.

\begin{proof}[Proof of Proposition \ref{prop_knorm}]
Let $c_0$ be the constant in Lemma \ref{lem8161941}, and fix $\kappa\in(0,c_0)$ and $k\geq1$. Recall from Section \ref{sec_statement} that $X_k$ is a closed real linear subspace of $L_\infty(\bR^d)$.
We first establish \eqref{eq8170137}.
For the second inequality in \eqref{eq8170137}, by Lemma \ref{lem8161941},
\begin{equation*}
e^{2^{2k}\kappa r} \|e^{rL_B}h\|_{L_\infty(\bR^d)} \leq N e^{-(c_0-\kappa)2^{2k}r} \|h\|_{L_\infty(\bR^d)}\leq N \|h\|_{L_\infty(\bR^d)}.
\end{equation*}
To prove the first inequality, let us fix $B_0:=\nu I\in\bA_{\nu,K}$. We claim that
\begin{equation}
\label{eq8170138}
\lim_{r\to0} \|e^{rL_{B_0}}h-h\|_{L_\infty(\bR^d)}=0.
\end{equation}
Since $\text{supp}(\cF[h])\subset \cA_k$,
\begin{equation*}
e^{rL_{B_0}}h-h=\cF^{-1}\left[\left(e^{-r\nu|\xi|^2}-1\right) \chi_k(\xi)\right]*h,
\end{equation*}
where $\chi_k\in C_c^\infty(\bR^d)$ such that $\chi_k=1$ on $\cA_k$.
As in \eqref{eq8170031}, for any integer $m>d/2$, one can show that
\begin{align*}
&\left\| \cF^{-1}\left[\left(e^{-r\nu|\xi|^2}-1\right) \chi_k(\xi)\right](\cdot) \right\|_{L_1(\bR^d)} \\
&\leq N \sum_{|\beta|\leq m} \left\|D_\xi^\beta\left[(e^{-r\nu |\xi|^2}-1)\chi_k(\xi)\right]\right\|_{L_2(\bR^d)}.
\end{align*}
By the dominated convergence theorem, we get \eqref{eq8170138}. Moreover, we have
\begin{equation*}
\|h\|_{L_\infty(\bR^d)} = \lim_{r\to0}\| e^{rL_{B_0}}h \|_{L_\infty(\bR^d)} \leq \|h\|_k.
\end{equation*}

Now we show that $(X_k,\|\cdot\|_k)$ is a Banach space.
Let $h_n$ be a Cauchy sequence in $(X_k,\|\cdot\|_k)$. By the first inequality in \eqref{eq8170137}, $h_n$ is Cauchy in $L_\infty(\bR^d)$. Since $X_k$ is closed in $L_\infty(\bR^d)$, there exists $h\in X_k$ such that $h_n\to h$ in $L_\infty(\bR^d)$. By the second inequality in \eqref{eq8170137},
\begin{equation*}
\|h_n-h\|_k \leq N\|h_n-h\|_{L_\infty(\bR^d)} \to 0.
\end{equation*}
Thus, $(X_k,\|\cdot\|_k)$ is a Banach space.
The proposition is proved.
\end{proof}

\begin{proof}[Proof of Theorem \ref{thm_frequency}]
Let
\begin{equation*}
w_k(t,x):=\Delta_k u(t,x), \qquad g_k(t,x):=\Delta_k f(t,x).
\end{equation*}
By \eqref{eq8170137} and the boundedness of the Littlewood-Paley projections, one can easily show that $w_k\in C^1([0,T];X_k)$ and $g_k\in L_\infty((0,T);X_k)$.
Since $a_{ij}$ are independent of $x$, we have
\begin{equation*}
\partial_t^\alpha w_k(t,\cdot) = \sum_{i,j=1}^d a_{ij}(t) D_{ij}w_k(t,\cdot)  +g_k(t,\cdot), \qquad w_k(0,\cdot) = 0,
\end{equation*}
in $X_k$ for almost every $t\in(0,T)$.

We first consider $k\geq1$. Let
\begin{equation*}
v_k(t):=\|w_k(t,\cdot)\|_k.
\end{equation*}
Since $w_k\in C^1([0,T];X_k)$, $v_k$ is Lipschitz continuous on $[0,T]$.
For $t\in(0,T)$ such that $v_k(t)>0$, choose $l_t\in X_k^*$ as in Lemma \ref{lem8170201}. Then we have
\begin{equation*}
\|l_t\|_{X_k^*}=1, \qquad l_t(w_k(t,\cdot))=v_k(t).
\end{equation*}
By Lemmas \ref{lem8170203} and \ref{lem8170202},
\begin{align*}
\partial_t^\alpha v_k(t) &\leq l_t(\partial_t^\alpha w_k(t,\cdot))
\\
&=l_t(L_{A(t)}w_k(t,\cdot))+l_t(g_k(t,\cdot))
\\
&\leq -2^{2k}\kappa v_k(t)+\|g_k(t,\cdot)\|_k,
\end{align*}
where $A(t)=(a_{ij}(t))$.
If $v_k(t)=0$, then the same inequality follows from Lemma \ref{lem8170203} with the zero functional. Hence, \eqref{eq_frequency} is proved.

Next, we consider $k=0$.
Using \eqref{eq8162343}, one can deduce that
\begin{equation} \label{eq8181044}
\|e^{rL_{A(t)}}h\|_{L_\infty(\bR^d)} \leq \|p_{A(t)}(r,\cdot)*h\|_{L_\infty(\bR^d)} \leq \|h\|_{L_\infty(\bR^d)}, \qquad h\in X_0.
\end{equation}
We claim that if $0\neq h\in X_0$ and $l_h\in X_0^*$ satisfies
\begin{equation*}
\|l_h\|_{X_0^*}=1, \qquad l_h(h) = \|h\|_{L_\infty(\bR^d)},
\end{equation*}
then
\begin{equation} \label{eq8181057}
l_h(L_{A(t)}h) \leq 0.
\end{equation}
Since $l_h$ is linear, \eqref{eq8181044} yields
\begin{equation*}
l_h\left(\frac{e^{rL_{A(t)}}h-h}{r}\right)\leq0, \qquad r>0.
\end{equation*}
Let $\chi_0\in C_c^\infty(\bR^d)$ such that $\chi_0=1$ on $\cA_0$. Since $\cF[h]=\chi_0\cF[h]$, we have
\begin{equation*}
\frac{e^{rL_{A(t)}}h-h}{r}-L_{A(t)}h=\cF^{-1}[m_r]*h,
\end{equation*}
where
\begin{equation*}
m_r(\xi):=\left(\frac{e^{-r\sum_{i,j=1}^da_{ij}(t)\xi_i\xi_j}-1}{r}+\sum_{i,j=1}^da_{ij}(t)\xi_i\xi_j\right)\chi_0(\xi).
\end{equation*}
For any multi-index $\beta$, $D_\xi^\beta m_r$ converges uniformly to zero as $r\to0$. Thus, as in \eqref{eq8170031},
\begin{equation*}
\frac{e^{rL_{A(t)}}h-h}{r}\to L_{A(t)}h \quad \text{ in } X_0,
\end{equation*}
which yields \eqref{eq8181057}.

Let
\begin{equation*}
v_0(t) := \|w_0(t,\cdot)\|_{L_\infty(\bR^d)}.
\end{equation*}
For a fixed $t\in(0,T]$ with $v_0(t)>0$, take a norming functional $l_t\in X_0^*$ for $w_0(t,\cdot)$. If $v_0(t)=0$, take $l_t=0$. Since
\begin{equation*}
\partial_t^\alpha w_0=L_{A(t)}w_0+g_0,
\end{equation*}
by Lemma \ref{lem8170203} and \eqref{eq8181057},
\begin{equation*}
\partial_t^\alpha v_0(t) \leq l_t(L_{A(t)}w_0(t,\cdot)) + l_t(g_0(t,\cdot))\leq\|g_0(t,\cdot)\|_{L_\infty(\bR^d)}.
\end{equation*}
This proves \eqref{eq_frequency_low}.
The proof is complete.
\end{proof}

\begin{example} \label{ex_variable_coefficients}
The frequency localized maximum principle fails even for analytic coefficients depending on the spatial variables.
For $x=(x_1,\cdots, x_d)\in\bR^d$, let us consider the operator
\begin{equation*}
    L(x): u\mapsto  (1+\varepsilon\cos(3x_1))\Delta u,
\end{equation*}
where $\varepsilon\in(0,1)$. 
Then one can verify that $L$ satisfies Assumption \ref{asscoeff}.
For $m=1,2$, we define $h_m(x) := \cos(mx_1)$. In $\cS'(\bR^d)$,
\begin{equation*}
    \cF[h_m](\xi) = c_d\left( \delta_{me_1}(\xi) + \delta_{-me_1}(\xi) \right),
\end{equation*}
where $e_1$ is the first coordinate vector and $\delta_\eta$ denotes the Dirac measure at $\eta \in \bR^d$.
Then by the definition of the Littlewood-Paley projection,
\begin{eqnarray} \label{eq10091118}
    &\Delta_1h_1=0, \quad \Delta_1(Lh_1)= -\frac{\varepsilon}{2}h_2,& \nonumber
    \\
    &\Delta_1 h_2=h_2, \quad \Delta_1(Lh_2) = -4h_2.&
\end{eqnarray}

We now construct a function $u$ for which \eqref{eq_frequency} fails when $k=1$.
Let
\begin{equation*}
    u(t,x) := b(t)h_1(x) -\frac{\varepsilon}{2}t^2 h_2(x),
\end{equation*}
where $b(t) := \frac{2t^{2-\alpha}}{\Gamma(3-\alpha)} + 4t^2$. Then by using \eqref{eq10091118}, one can compute
\begin{equation*}
    \Delta_1 u (t,\cdot) = -\frac{\varepsilon}{2}t^2 h_2, \qquad \Delta_1 f(t,\cdot) = \frac{\varepsilon}{2} \left( b(t)-\partial_t^\alpha(t^2) - 4t^2 \right) h_2 =0,
\end{equation*}
where $f:=\partial_t^\alpha u - Lu$. Thus,
\begin{equation*}
    \partial_t^\alpha \|\Delta_1 u(t,\cdot)\|_{1} = \frac{\varepsilon t^{2-\alpha}}{\Gamma(3-\alpha)} \|h_2\|_{1} >0,
\end{equation*}
whereas $\|\Delta_1 f\|_{1}=0$. Hence, \eqref{eq_frequency} fails for $k=1$.
We also remark that the same example can be applied to the classical heat equation with variable coefficients when $\alpha=1$.
\end{example}

\section{A priori estimate} \label{sec_apriori}

In this section, we prove the a priori estimate \eqref{eq_main_est}. We first treat coefficients that are independent of the spatial variable in Section \ref{sec_simple} and then consider general coefficients in Section \ref{sec_general}.

\subsection{Simple coefficients} \label{sec_simple}

Let
\begin{equation*}
    E_{\alpha,\beta}(z) :=\sum_{n=0}^\infty \frac{z^n}{\Gamma(\alpha n+\beta)}
\end{equation*}
be the two-parameter Mittag-Leffler function.

\begin{lemma} \label{lem_max}
Let $\alpha\in(0,1)$, $T\in(0,\infty)$, $c\in[0,\infty)$, and $v\in C^1([0,T])$. Suppose that $v(0)\leq0$ and
\begin{equation}\label{eq8160138}
    \partial_t^\alpha v(t)+cv(t)\leq0, \quad t\in(0,T].
\end{equation}
Then $v(t)\leq0$ for every $t\in[0,T]$.
\end{lemma}

\begin{proof}
Assume that $v$ has a positive maximum at some $t_0\in(0,T]$. Then $v(t_0)-v(0)>0$ and $v(t_0)-v(s)\geq0$ for every $s\in[0,t_0]$. 
The relation \eqref{eqMarch} easily yields
\begin{equation*}
    \partial_t^\alpha v(t_0)>0.
\end{equation*}
Since $cv(t_0)\geq0$, \eqref{eq8160138} leads to a contradiction. Thus, $v\leq0$ on $[0,T]$. The lemma is proved.
\end{proof}

\begin{corollary} \label{cor_comparison}
Let $\alpha\in(0,1)$, $T\in(0,\infty)$, $c\in[0,\infty)$, and $h\in L_\infty((0,T))$. Suppose that $v$ is an absolutely continuous function on $[0,T]$ such that $v(0)\leq0$ and
\begin{equation} \label{eq8252324}
    \partial_t^\alpha v(t)+cv(t)\leq h(t),
\end{equation}
for almost every $t\in(0,T)$.
Then
\begin{align} \label{eq8260010}
    v(t) \leq \int_0^t (t-s)^{\alpha-1} E_{\alpha,\alpha}(-c(t-s)^\alpha)h(s) ds,\quad t\in[0,T].
\end{align}
\end{corollary}

\begin{proof}
Let $\varepsilon>0$, and let $\varphi_\varepsilon$ be the mollifier introduced in Lemma \ref{lem_dense}.
Let us extend $v$ and $h$ by zero to $(-\infty,0)$. Then
\begin{align*}
\partial_t^\alpha v^{(\varepsilon)} &= \frac{d}{dt} \left( k_{1-\alpha}*(v^{(\varepsilon)}-v^{(\varepsilon)}(0)) \right)
\\
&=(\partial_t^\alpha v)^{(\varepsilon)} + v(0)(\varphi_\varepsilon*k_{1-\alpha}),
\end{align*}
where $k_{1-\alpha}(t):= \frac{t^{-\alpha}}{\Gamma(1-\alpha)}1_{(0,\infty)}(t)$.
This together with \eqref{eq8252324} leads to
\begin{equation*}
    \partial_t^\alpha v^{(\varepsilon)}(t)+cv^{(\varepsilon)}(t)\leq h^{(\varepsilon)}(t),
\end{equation*}
where $v^{(\varepsilon)}$ and $h^{(\varepsilon)}$ are defined as in \eqref{eq8252326}.

By \cite[Theorem 4.1]{KST06},
\begin{equation*}
    w^\varepsilon(t):= \int_0^t (t-s)^{\alpha-1} E_{\alpha,\alpha}(-c(t-s)^\alpha)h^{(\varepsilon)}(s) ds
\end{equation*}
is a solution to
\begin{equation*}
    \partial_t^\alpha w^\varepsilon(t) + cw^\varepsilon(t) = h^{(\varepsilon)}(t)
\end{equation*}
with $w^\varepsilon(0)=0$. 
Note that since $h^{(\varepsilon)}(s)=0$ for $s<0$, by a change of variables,
\begin{align*}
    w^\varepsilon(t) = \int_0^T s^{\alpha-1} E_{\alpha,\alpha}(-cs^\alpha)h^{(\varepsilon)}(t-s) ds.
\end{align*}
Thus, by using
\begin{equation} \label{eq8252345}
    0\leq E_{\alpha,\alpha}(-r)\leq\frac{N}{(1+r)^2}, \qquad r\geq0
\end{equation}
(see e.g. \cite[Section 6]{HMS11}),
\begin{equation*}
    \frac{d}{dt}w^\varepsilon(t) = \int_0^T s^{\alpha-1} E_{\alpha,\alpha}(-cs^\alpha) \frac{d}{dt}h^{(\varepsilon)}(t-s) ds,
\end{equation*}
which yields $w^\varepsilon\in C^1([0,T])$.
Therefore, we can apply Lemma \ref{lem_max} to $v^{(\varepsilon)}-w^\varepsilon$ and obtain that 
\begin{align*}
    v^{(\varepsilon)}(t) \leq \int_0^t (t-s)^{\alpha-1} E_{\alpha,\alpha}(-c(t-s)^\alpha)h^{(\varepsilon)}(s) ds.
\end{align*}
Since $v$ is continuous in $(0,T)$ and $h$ is bounded, letting $\varepsilon\to0$ in this inequality yields \eqref{eq8260010} for $t\in(0,T)$. We note that the case $t=0$ is trivial.
The proof is completed.
\end{proof}

\begin{lemma}
Let $T\in(0,\infty)$, $\alpha\in(0,1)$, $\gamma\in(0,\infty)$, and $\kappa\in(0,c_0)$ be fixed, where $c_0$ is as in Lemma \ref{lem8161941}. 
Suppose that Assumption \ref{asscoeff} holds, $a_{ij}=a_{ij}(t)$ are independent of $x$, and that $b_i=c=0$. Assume that $u \in C^1([0,T];\Lambda^{\gamma+2}(\bR^d))$ is a solution to \eqref{eq_intro} with $u(0,\cdot)=0$.

(i) For every $k\geq1$ and $t\in[0,T]$,
\begin{align} 
    &\|\Delta_ku(t,\cdot)\|_{L_\infty(\bR^d)}\notag\\
    &\leq N\int_0^t(t-s)^{\alpha-1}E_{\alpha,\alpha}\left(-2^{2k}\kappa(t-s)^\alpha\right)\|\Delta_kf(s,\cdot)\|_{L_\infty(\bR^d)}ds,\label{eq8170208}
\end{align}
where $N$ depends only on $d$, $\nu$, $K$, and $\psi$.

(ii) When $k=0$, we have
\begin{equation} \label{eq8261856}
    \|\Delta_0u(t,\cdot)\|_{L_\infty(\bR^d)}\leq \frac{1}{\Gamma(\alpha)}\int_0^t(t-s)^{\alpha-1}\|\Delta_0f(s,\cdot)\|_{L_\infty(\bR^d)}ds.
\end{equation}
\end{lemma}

\begin{proof}
For $k\in\bN_0$, denote
\begin{equation*}
v_k(t) := \|\Delta_k u(t,\cdot)\|_k.
\end{equation*}

$(i)$ For $k\geq1$, by \eqref{eq_frequency} and Corollary \ref{cor_comparison} with $c=2^{2k}\kappa$,
\begin{equation*}
v_k(t)\leq\int_0^t(t-s)^{\alpha-1}E_{\alpha,\alpha}\left(-2^{2k}\kappa(t-s)^\alpha\right)\|\Delta_k f(s,\cdot)\|_k ds.
\end{equation*}
Thus, \eqref{eq8170208} follows from \eqref{eq8170137} for both $\Delta_k u$ and $\Delta_k f$.

$(ii)$ For $k=0$, by \eqref{eq_frequency_low} and Corollary \ref{cor_comparison} with $c=0$ and $E_{\alpha,\alpha}(0)=1/\Gamma(\alpha)$,
\begin{equation*}
\|\Delta_0 u(t,\cdot)\|_{L_\infty(\bR^d)}\leq\frac{1}{\Gamma(\alpha)}\int_0^t(t-s)^{\alpha-1}\|\Delta_0 f(s,\cdot)\|_{L_\infty(\bR^d)} ds,
\end{equation*}
which yields \eqref{eq8261856}.
The lemma is proved.
\end{proof}

\begin{lemma}
Let $T\in(0,\infty)$, $\alpha\in(0,1)$, $\gamma\in(0,\infty)$, and $p\in(1,\infty]$. Let $\omega\in A_p(\bR)$ with $[\omega]_{A_p(\bR)}\leq K_0$ (put $\omega\equiv1$ and $K_0=1$ when $p=\infty$).
Suppose that Assumption \ref{asscoeff} holds, $a_{ij}=a_{ij}(t)$ are independent of $x$, and that $b_i=c=0$. Assume that $u\in \bH_{p,\omega,0}^{\alpha,\gamma+2}(T)$ satisfies \eqref{eq_intro}. Then we have
\begin{equation} \label{eq8171625}
    \|u(t,\cdot)\|_{\Lambda^{\gamma+2}(\bR^d)} \leq N(T) \bM_t\left(\|f(\star,\cdot)\|_{\Lambda^{\gamma}(\bR^d)} 1_{(0,T)}(\star)\right)(t),
\end{equation}
for almost every $t\in(0,T)$, where $N(T):=N(1+T^\alpha)$ and $N$ depends only on $d$, $\alpha$, $\gamma$, $\nu$, $K$, and $\psi$. In particular, $N(T)$ is increasing in $T$.
\end{lemma}

\begin{proof}
First, we prove \eqref{eq8171625} when $u\in C^1([0,T];\Lambda^{\gamma+2}(\bR^d))$ and $u(0,\cdot)=0$.
Let
\begin{equation*}
    F(s) := \|f(s,\cdot)\|_{\Lambda^{\gamma}(\bR^d)} 1_{(0,T)}(s),
\end{equation*}
and extend $F$ by zero outside $(0,T)$. For $k\geq1$, by \eqref{eq8170208},
\begin{equation}\label{eq8170211}
    2^{(\gamma+2)k}\|\Delta_ku(t,\cdot)\|_{L_\infty(\bR^d)}\leq N\int_0^tK_k(t-s)F(s)ds,
\end{equation}
where $K_k(r):=2^{2k}r^{\alpha-1}E_{\alpha,\alpha}\left(-\kappa2^{2k}r^\alpha\right)$.
Due to \eqref{eq8252345} we have
\begin{equation} \label{eq8170213}
    0\leq K_k(r)\leq N\frac{2^{2k}r^{\alpha-1}}{(1+2^{2k}r^\alpha)^2}.
\end{equation}
Here, we note that the dependence on $\kappa$ can be removed by simply taking $\kappa=c_0/2$.
Since $F=0$ outside $(0,T)$,
\begin{equation} \label{eq8170214}
    \int_0^t K_k(t-s)F(s)ds = \sum_{n\in\bZ} \int_{t-2^n}^{t-2^{n-1}} K_k(t-s)F(s)ds.
\end{equation}
By \eqref{eq8170213}, for $s\in (t-2^n,t-2^{n-1})$,
\begin{equation*}
    K_k(t-s)\leq N2^{-n}\frac{2^{2k+\alpha n}}{(1+2^{2k+\alpha n})^2}.
\end{equation*}
Thus,
\begin{align} \label{eq8170216}
    \int_{t-2^n}^{t-2^{n-1}}K_k(t-s)F(s)ds &\leq N \frac{2^{2k+\alpha n}}{(1+2^{2k+\alpha n})^2} \aint_{t-2^n}^{t-2^{n-1}} F(s) ds \nonumber
    \\
    &\leq N \frac{2^{2k+\alpha n}}{(1+2^{2k+\alpha n})^2} \bM_tF(t).
\end{align}
Choose $n_k\in\bZ$ so that $2^{2k+\alpha n_k}\leq1<2^{2k+\alpha(n_k+1)}$, and put $a_k=2^{2k+\alpha n_k}$. 
Then $2^{-\alpha} < a_k \leq 1$. Since
\begin{equation*}
    \frac{x}{(1+x)^2}\leq\min\{x,x^{-1}\},
\end{equation*}
we obtain
\begin{equation}\label{eq8170217}
    \sum_{n\in\bZ}\frac{2^{2k+\alpha n}}{(1+2^{2k+\alpha n})^2}=\sum_{j\in\bZ}\frac{a_k2^{\alpha j}}{(1+a_k2^{\alpha j})^2}\leq\sum_{j\leq0}2^{\alpha j}+2^\alpha\sum_{j\geq1}2^{-\alpha j}\leq N,
\end{equation}
where $N$ is independent of $k$. Combining \eqref{eq8170211}, \eqref{eq8170214}, \eqref{eq8170216}, and \eqref{eq8170217}, we obtain
\begin{equation} \label{eq8181115}
    2^{(\gamma+2)k}\|\Delta_ku(t,\cdot)\|_{L_\infty(\bR^d)}\leq N\bM_tF(t).
\end{equation}

We now estimate $\Delta_0u$.
Let $n_t\in\bZ$ such that $2^{n_t-1}<t\leq2^{n_t}$. By the change of variables,
\begin{align} \label{eq9231819}
    \int_0^t(t-s)^{\alpha-1}F(s)ds &\leq\sum_{n\leq n_t}\int_{2^{n-1}}^{\min\{2^n,t\}} s^{\alpha-1}F(t-s) ds \nonumber
    \\
    &\leq N\sum_{n\leq n_t} 2^{\alpha n} \bM_tF(t)\leq N t^\alpha \bM_tF(t).
\end{align}
By \eqref{eq8261856}, we have
\begin{equation} \label{eq8181114}
    \|\Delta_0u(t,\cdot)\|_{L_\infty(\bR^d)}\leq NT^\alpha\bM_tF(t).
\end{equation}
Hence, by \eqref{def_hol}, \eqref{eq8181115} and \eqref{eq8181114}, we obtain \eqref{eq8171625} when $u\in C^1([0,T];\Lambda^{\gamma+2}(\bR^d))$.

Next, we consider general $u\in \bH_{p,\omega,0}^{\alpha,\gamma+2}(T)$. Let $u^{(\varepsilon)}$ be as in \eqref{eq8252326}, and let $g^{(\varepsilon)} := \partial_t^\alpha u^{(\varepsilon)} = (\partial_t^\alpha u)^{(\varepsilon)}$. Then $u^{(\varepsilon)}\in C^1([0,T];\Lambda^{\gamma+2}(\bR^d))$ satisfies
\begin{equation*}
    \partial_t^\alpha u^{(\varepsilon)}(t,x) = \sum_{i,j=1}^d a_{ij}(t)D_{ij}u^{(\varepsilon)}(t,x) + f^\varepsilon(t,x),
\end{equation*}
where
\begin{equation*}
    f^\varepsilon := g^{(\varepsilon)} - \sum_{i,j=1}^d a_{ij}(t)D_{ij}u^{(\varepsilon)}.
\end{equation*}
Since $u^{(\varepsilon)}\in C^1([0,T];\Lambda^{\gamma+2}(\bR^d))$, by the above result,
\begin{equation} \label{eq9142356}
        \|u^{(\varepsilon)}(t,\cdot)\|_{\Lambda^{\gamma+2}(\bR^d)} \leq N(T) \bM_t\left(\|f^\varepsilon(\star,\cdot)\|_{\Lambda^{\gamma}(\bR^d)} 1_{(0,T)}(\star)\right)(t).
\end{equation}
If $p<\infty$, then by Lemma \ref{lem_dense}, $u^{(\varepsilon)}\to u$ in $\mathbf{\Lambda}_{p,\omega}^{\gamma+2}(T)$. Thus, there is a sequence $\varepsilon_n\to0$ such that
\begin{equation} \label{eq8262013}
    \|u^{(\varepsilon_n)}(t,\cdot)\|_{\Lambda^{\gamma+2}(\bR^d)} \to \|u(t,\cdot)\|_{\Lambda^{\gamma+2}(\bR^d)}
\end{equation}
as $n\to\infty$ for almost every $t\in(0,T)$. When $p=\infty$, since $\mathbf{\Lambda}_{\infty}^{\gamma+2}(T) \subset \mathbf{\Lambda}_{q,1}^{\gamma+2}(T)$ for any $q\in(1,\infty)$, we still obtain \eqref{eq8262013}.

 We also show that there is a subsequence $\varepsilon_{n_k}$ of $\varepsilon_n$ such that
\begin{equation} \label{eq8262016}
    \bM_t\left(\|f^{\varepsilon_{n_k}}(\star,\cdot)\|_{\Lambda^{\gamma}(\bR^d)} 1_{(0,T)}(\star)\right)(t) \to \bM_t\left(\|f(\star,\cdot)\|_{\Lambda^{\gamma}(\bR^d)} 1_{(0,T)}(\star)\right)(t)
\end{equation}
as $n_k\to\infty$ for almost every $t\in(0,T)$.
As above, when $p=\infty$, $\bH_{\infty,0}^{\alpha,\gamma+2}(T) \subset \bH_{q,1,0}^{\alpha,\gamma+2}(T)$ for any $q\in(1,\infty)$. Thus, one just needs to prove \eqref{eq8262016} when $p<\infty$.
In this case, by \eqref{eq8251608},
\begin{equation*}
    f^{\varepsilon_{n}}\to \partial_t^\alpha u - \sum_{i,j=1}^d a_{ij}(t)D_{ij}u = f
\end{equation*}
in $\mathbf{\Lambda}_{p,\omega}^{\gamma}(T)$ as $\varepsilon_{n}\to0$. Therefore, by the weighted Hardy-Littlewood maximal function theorem,
\begin{equation*}
    \left\|\bM_t\left(\|(f^{\varepsilon_{{n}}} - f)(\star,\cdot)\|_{\Lambda^{\gamma}(\bR^d)} 1_{(0,T)}(\star)\right)(t) \right\|_{L_p((0,T),\omega dt)} \to 0
\end{equation*}
as $\varepsilon_{n} \to 0$, which allows us to extract a further subsequence satisfying \eqref{eq8262016}.

Lastly, to conclude \eqref{eq8171625}, it suffices to replace $\varepsilon$ by $\varepsilon_{n_k}$ in \eqref{eq9142356} and let $n_k\to\infty$.
Combining \eqref{eq8262013} and \eqref{eq8262016}, we obtain \eqref{eq8171625}.
The lemma is proved.
\end{proof}

\subsection{General coefficients} \label{sec_general}

We first state some auxiliary tools, which are well known; see, for instance, \cite{CR26}.

\begin{lemma} \label{lem_aux}
    Let $\gamma\in(0,\infty)$.

    (i) Let $\zeta\in C_c^\infty(\bR^d)$ satisfy $\zeta=1$ on $B_1$, $\zeta=0$ on $\bR^d\setminus B_2$, and $0\leq\zeta\leq1$, and denote $\zeta_r^y(x):=\zeta((x-y)/r)$.
Then for $r>0$,
\begin{equation} \label{locali}
\|h\|_{\Lambda^\gamma(\bR^d)} \leq N(d,r,\zeta,\gamma) \sup_{y\in\bR^d} \|h\zeta_r^y\|_{\Lambda^\gamma(\bR^d)}.
\end{equation}

(ii) For $u\in \Lambda^1(\bR^d)$,
\begin{equation} \label{eqlog}
    \sup_{\substack{|x-y|<1 \\ x\neq y}} \frac{|u(x)-u(y)|}{|x-y|\log_2(2/|x-y|)} \leq N(d)\|u\|_{\Lambda^1(\bR^d)}.
\end{equation}

(iii) For $\varepsilon>0$ and $u\in \Lambda^{\gamma+2}(\bR^d)$,
\begin{align} \label{eq_inter}
\|u\|_{\Lambda^{\gamma+1}(\bR^d)} &\leq \varepsilon \|u\|_{\Lambda^{\gamma+2}(\bR^d)} + N(\varepsilon,d,\gamma) \|u\|_{\Lambda^\gamma(\bR^d)}, \nonumber
\\
\|D^2u\|_{L_\infty(\bR^d)} &\leq \varepsilon\|u\|_{\Lambda^{\gamma+2}(\bR^d)} + N(\varepsilon,d,\gamma) \|u\|_{L_\infty(\bR^d)}.
\end{align}

(iv) For $u,v\in \Lambda^\gamma(\bR^d)$,
\begin{equation} \label{eqprod}
\|uv\|_{\Lambda^\gamma(\bR^d)} \leq N \left( \|u\|_{L_\infty(\bR^d)} \|v\|_{\Lambda^\gamma(\bR^d)} + \|u\|_{\Lambda^\gamma(\bR^d)} \|v\|_{L_\infty(\bR^d)} \right),
\end{equation}
where $N=N(d,\gamma)$.
\end{lemma}

\begin{lemma}\label{lem8080013}
Let $0\leq S_0<S_1<\infty$, $\alpha\in(0,1)$, $\gamma\in(0,\infty)$, and $p\in(1,\infty]$. Let $\omega\in A_p(\bR)$ with $[\omega]_{A_p(\bR)}\leq K_0$ (put $\omega\equiv1$ and $K_0=1$ when $p=\infty$). Then for $v\in \bH_{p,\omega,0}^{\alpha,\gamma+2}(S_0,S_1)$,
\begin{equation} \label{eq8180015}
\|v\|_{\mathbf{\Lambda}_{p,\omega}^{\gamma}(S_0,S_1)} \leq N(S_1-S_0)^\alpha\|\partial_{t,S_0}^\alpha v\|_{\mathbf{\Lambda}_{p,\omega}^{\gamma}(S_0,S_1)},
\end{equation}
where $N$ depends only on $\alpha$, $p$, and $K_0$.
\end{lemma}

\begin{proof}
By \eqref{eq9290128}, $I_{S_0}^{\alpha}\partial_{t,S_0}^{\alpha} v =v$.
In other words,
\begin{equation*}
v(t) = \frac{1}{\Gamma(\alpha)}\int_{S_0}^t(t-s)^{\alpha-1}\partial_{t,S_0}^\alpha v(s)ds.
\end{equation*}
Thus, \eqref{eq8180015} follows from \eqref{eq8241430}.
The lemma is proved.
\end{proof}

\begin{lemma} \label{lem_apriori}
Let the assumptions of Theorem \ref{thm_main} be satisfied. Then for any $u\in\bH_{p,\omega,0}^{\alpha,\gamma+2}(T)$ satisfying \eqref{eq_intro}, we have
\begin{equation} \label{est_apriori}
\|\partial_t^\alpha u\|_{\mathbf{\Lambda}_{p,\omega}^\gamma(T)} + \|u\|_{\mathbf{\Lambda}_{p,\omega}^{\gamma+2}(T)} \leq N \|f\|_{\mathbf{\Lambda}_{p,\omega}^\gamma(T)},
\end{equation}
where $N=N(d,\alpha,\gamma,p,K_0,\nu,K,T)$.
\end{lemma}

\begin{proof}
We divide the proof into two steps.

\textbf{Step 1.}
In this step, we aim to prove that, for every $0\leq S_0<S_1\leq T$,
\begin{equation} \label{eq8180023}
\|\partial_{t,S_0}^\alpha v\|_{\mathbf{\Lambda}_{p,\omega}^\gamma(S_0,S_1)} + \|v\|_{\mathbf{\Lambda}_{p,\omega}^{\gamma+2}(S_0,S_1)} \leq N \big(\|g\|_{\mathbf{\Lambda}_{p,\omega}^\gamma(S_0,S_1)} + \|v\|_{\mathbf{\Lambda}_{p,\omega}^\gamma(S_0,S_1)}\big),
\end{equation}
whenever $v\in \bH_{p,\omega,0}^{\alpha,\gamma+2}(S_0,S_1)$ is a solution to
\begin{equation} \label{eq9291338}
\partial_{t,S_0}^\alpha v=\sum_{i,j=1}^da_{ij}D_{ij}v+\sum_{i=1}^db_iD_iv+cv+g.
\end{equation}
We emphasize that the constant in \eqref{eq8180023} is independent of $S_0$ and $S_1$.

We first assume that $b_i=c=0$. Fix $y\in\bR^d$, and let $r\in(0,1/4]$ and $\zeta_r^y$ be taken from Lemma \ref{lem_aux} $(i)$. Then $v^y:=v\zeta_r^y$ satisfies
\begin{equation*}
\partial_{t,S_0}^\alpha v^y(t,x) = \sum_{i,j=1}^da_{ij}(t,y) D_{ij}v^y(t,x) + F^y(t,x),
\end{equation*}
where
\begin{align*}
F^y(t,x) &:= \sum_{i,j=1}^d\big(a_{ij}(t,x)-a_{ij}(t,y)\big) D_{ij}v^y(t,x)+G^y(t,x)
\end{align*}
and
\begin{align*}
G^y(t,x) &:= \zeta_r^y(x) g(t,x) -\sum_{i,j=1}^d (a_{ij}(t,x) + a_{ji}(t,x)) D_iv(t,x)D_j \zeta_r^y(x)
\\
&\quad- \sum_{i,j=1}^da_{ij}(t,x)v(t,x) D_{ij} \zeta_r^y(x).
\end{align*}
Let $\chi\in C_c^\infty(B_4)$ satisfy $\chi=1$ on $B_3$. Then $\chi_r^y(x) := \chi((x-y)/r)\in C_c^\infty(B_{4r}(y))$ and $\chi_r^y=1$ on $B_{3r}(y)$. Let
\begin{equation*}
\rho_\gamma(r) := r^\gamma1_{\{0<\gamma<1\}}+r\log(2/r)1_{\{\gamma=1\}}+r1_{\{\gamma>1\}}.
\end{equation*}
Then by \eqref{eqlog} and \eqref{eqprod},
\begin{align} \label{eq8080042}
\|\big(a_{ij}(t,\cdot)-a_{ij}(t,y)\big)\chi_r^y\|_{L_\infty(\bR^d)} &\leq \rho_\gamma(r) N(K), \nonumber
\\
\|\big(a_{ij}(t,\cdot)-a_{ij}(t,y)\big)\chi_r^y\|_{\Lambda^\gamma(\bR^d)} &\leq N(r,K).
\end{align}
Since $D_{ij}v^y$ is supported in $B_{2r}(y)$, \eqref{eq_inter}, \eqref{eqprod}, and \eqref{eq8080042} imply that for any $\varepsilon>0$,
\begin{align*}
&\left\|\sum_{i,j=1}^d\big(a_{ij}(t,\cdot)-a_{ij}(t,y)\big)D_{ij}v^y(t,\cdot)\right\|_{\Lambda^\gamma(\bR^d)} 
\\
&\leq N\rho_\gamma(r)\|v^y(t,\cdot)\|_{\Lambda^{\gamma+2}(\bR^d)} + N(r) \|D^2v^y(t,\cdot)\|_{L_\infty(\bR^d)}
\\
&\leq (N\rho_\gamma(r)+N(r)\varepsilon)\|v^y(t,\cdot)\|_{\Lambda^{\gamma+2}(\bR^d)} + N(r,\varepsilon) \|v(t,\cdot)\|_{\Lambda^\gamma(\bR^d)}.
\end{align*}
Here, for the second inequality, we used
\begin{equation*}
    \|v^y(\star,\cdot)\|_{\Lambda^{\gamma'}(\bR^d)} \leq N(r)\|v(\star,\cdot)\|_{\Lambda^{\gamma'}(\bR^d)}, \quad \gamma'>0,
\end{equation*}
which follows from \eqref{eqprod}.
Thus, by taking $\varepsilon>0$ so that $N(r)\varepsilon\leq N\rho_\gamma(r)$,
\begin{align} \label{eq9291239}
    &\left\|\sum_{i,j=1}^d\big(a_{ij}(t,\cdot)-a_{ij}(t,y)\big)D_{ij}v^y(t,\cdot)\right\|_{\Lambda^\gamma(\bR^d)} \nonumber
\\
&\leq N\rho_\gamma(r) \|v^y(t,\cdot)\|_{\Lambda^{\gamma+2}(\bR^d)} + N(r) \|v(t,\cdot)\|_{\Lambda^\gamma(\bR^d)}.
\end{align}
Again by \eqref{eqprod},
\begin{equation*}
\|G^y(t,\cdot)\|_{\Lambda^\gamma(\bR^d)} \leq N(r) \big(\|g(t,\cdot)\|_{\Lambda^\gamma(\bR^d)}+\|v(t,\cdot)\|_{\Lambda^{\gamma+1}(\bR^d)} + \|v(t,\cdot)\|_{\Lambda^\gamma(\bR^d)}\big).
\end{equation*}
This and \eqref{eq9291239} imply that $F^y\in \mathbf{\Lambda}_{p,\omega}^\gamma(S_0,S_1)$. Thus, after translating the interval $(S_0,S_1)$ to $(0,S_1-S_0)$, we can apply \eqref{eq8171625} to obtain
\begin{align*}
\|v^y(t,\cdot)\|_{\Lambda^{\gamma+2}(\bR^d)} &\leq N \rho_\gamma(r) \bM_t\left( \|v^y(\star,\cdot)\|_{\Lambda^{\gamma+2}(\bR^d)} 1_{(S_0,S_1)}(\star) \right)(t) 
\\
&\quad+ N(r) \bM_t\left(\|g(\star,\cdot)\|_{\Lambda^{\gamma}(\bR^d)} 1_{(S_0,S_1)}(\star) \right)(t) 
\\
&\quad+ N(r)\bM_t\left( \|v(\star,\cdot)\|_{\Lambda^{\gamma+1}(\bR^d)} 1_{(S_0,S_1)}(\star) \right)(t) 
\\
&\quad+ N(r) \bM_t\left( \|v(\star,\cdot)\|_{\Lambda^{\gamma}(\bR^d)} 1_{(S_0,S_1)}(\star) \right)(t)
\end{align*}
for almost every $t\in (S_0,S_1)$.
By taking the supremum over $y\in\bR^d$ and using the weighted Hardy-Littlewood maximal function theorem,
\begin{align*}
\|V\|_{L_p((S_0,S_1),\omega dt)} &\leq N \rho_\gamma(r) \|V\|_{L_p((S_0,S_1),\omega dt)}
\\
&\quad+ N(r) \left( \|g\|_{\mathbf{\Lambda}_{p,\omega}^\gamma(S_0,S_1)} + \|v\|_{\mathbf{\Lambda}_{p,\omega}^{\gamma+1}(S_0,S_1)} + \|v\|_{\mathbf{\Lambda}_{p,\omega}^\gamma(S_0,S_1)} \right),
\end{align*}
where $V(t)=\sup_{y\in\bR^d}\|v^y(t,\cdot)\|_{\Lambda^{\gamma+2}(\bR^d)}$. Here, we note that since $y \to v^y=v\zeta_r^y$ is continuous in $\Lambda^{\gamma+2}(\bR^d)$, the supremum over $y\in\bR^d$ can be taken over a countable dense subset of $\bR^d$, and thus $V$ is measurable.
By choosing $r=r_0\in(0,1/4]$ sufficiently small so that $N\rho_\gamma(r_0)\leq1/2$, \eqref{locali} leads to
\begin{equation*}
\|v\|_{\mathbf{\Lambda}_{p,\omega}^{\gamma+2}(S_0,S_1)} \leq N \big(\|g\|_{\mathbf{\Lambda}_{p,\omega}^\gamma(S_0,S_1)} + \|v\|_{\mathbf{\Lambda}_{p,\omega}^{\gamma+1}(S_0,S_1)} + \|v\|_{\mathbf{\Lambda}_{p,\omega}^\gamma(S_0,S_1)}\big).
\end{equation*}
Now the estimate for $\|v\|_{\mathbf{\Lambda}_{p,\omega}^{\gamma+2}(S_0,S_1)}$ follows from \eqref{eq_inter}. To conclude \eqref{eq8180023}, one just needs to use \eqref{eq9291338}.

Next, we deal with general $b_i$ and $c$. By \eqref{eq_inter} and \eqref{eqprod},
\begin{equation*}
\left\| \sum_{i=1}^db_iD_iv+cv \right\|_{\mathbf{\Lambda}_{p,\omega}^\gamma(S_0,S_1)} \leq \varepsilon \|v\|_{\mathbf{\Lambda}_{p,\omega}^{\gamma+2}(S_0,S_1)} + N(\varepsilon) \|v\|_{\mathbf{\Lambda}_{p,\omega}^\gamma(S_0,S_1)}.
\end{equation*}
Thus, if we apply \eqref{eq8180023} with $g+\sum_i b_iD_iv+cv$ in place of $g$, and choose $\varepsilon$ sufficiently small, then we obtain \eqref{eq8180023}.

\textbf{Step 2.}
By \eqref{eq8180023} with $S_0=0$ and $S_1=T$, we have
\begin{equation*}
    \|\partial_{t}^\alpha u\|_{\mathbf{\Lambda}_{p,\omega}^\gamma(T)} + \|u\|_{\mathbf{\Lambda}_{p,\omega}^{\gamma+2}(T)} \leq N \big(\|f\|_{\mathbf{\Lambda}_{p,\omega}^\gamma(T)} + \|u\|_{\mathbf{\Lambda}_{p,\omega}^\gamma(T)}\big).
\end{equation*}
In this step, our goal is to estimate $\|u\|_{\mathbf{\Lambda}_{p,\omega}^\gamma(T)}$.

For an integer $m\geq2$, denote $\delta:=T/m$ and $s_k:=k\delta$ for $k=0,1,\ldots,m$. Here, $m$ will be determined later.
By Lemma \ref{lem8080013} and \eqref{eq8180023} with $(S_0,S_1)=(0,s_1)$,
\begin{equation} \label{eq8270027}
\|u\|_{\mathbf{\Lambda}_{p,\omega}^\gamma(0,s_1)} \leq N\delta^\alpha \big(\|f\|_{\mathbf{\Lambda}_{p,\omega}^\gamma(0,s_1)} + \|u\|_{\mathbf{\Lambda}_{p,\omega}^\gamma(0,s_1)}\big).
\end{equation}

For $k=1,\cdots,m-1$, take $\eta_k\in C^\infty(\bR)$ such that $\eta_k=0$ on $(-\infty,s_{k-1}]$, $\eta_k=1$ on $[s_k,\infty)$, and $|\eta_k'|\leq N\delta^{-1}$. Extend $u$ by zero for $t<0$ and put $u_k:=u\eta_k$. Then one can verify that $u_k \in \bH_{p,\omega,0}^{\alpha,\gamma+2}(s_{k-1},s_{k+1})$.
Moreover, as in the proof of \cite[Lemma 3.6]{DK19},
\begin{equation*} 
\partial_{t,s_{k-1}}^\alpha u_k = \sum_{i,j=1}^d a_{ij} D_{ij} u_k+\sum_{i=1}^d b_iD_i u_k + cu_k + f\eta_k + h_k,
\end{equation*}
where
\begin{align*}
h_k(t,x) &:= \frac{\alpha}{\Gamma(1-\alpha)}\int_{-\infty}^t(t-s)^{-1-\alpha}\big(\eta_k(t)-\eta_k(s)\big)u(s,x) ds
\\
&= \frac{\alpha}{\Gamma(1-\alpha)}\int_{-\infty}^t(t-s)^{-1-\alpha}\big(\eta_k(t)-\eta_k(s)\big)u(s,x) 1_{(0,s_k)}(s) ds.
\end{align*}
Here, the last equality holds since, when $t>s_k$, $\eta_k(t)-\eta_k(s)=0$ for $s>s_k$.
Using
\begin{equation*}
|\eta_k(t)-\eta_k(s)|\leq N\min\{1,|t-s|/\delta\},
\end{equation*}
we have
\begin{align*}
    \|h_k(t,\cdot)\|_{\Lambda^\gamma(\bR^d)} &\leq N \delta^{-1} \int_{t-\delta}^{t} |t-s|^{-\alpha} \|u(s,\cdot)\|_{\Lambda^\gamma(\bR^d)} 1_{(0,s_k)}(s) ds
    \\
    &\quad+ N\int_{-\infty}^{t-\delta} |t-s|^{-1-\alpha} \|u(s,\cdot)\|_{\Lambda^\gamma(\bR^d)} 1_{(0,s_k)}(s) ds
    \\
    &=: J_1 + J_2.
\end{align*}
Actually, $J_1(t) = N \delta^{-1} I_{t-\delta}^{1-\alpha}(\|u(\star,\cdot)\|_{\Lambda^\gamma(\bR^d)}1_{(0,s_k)}(\star))(t)$. Thus, by \eqref{eq8241429},
\begin{equation} \label{eq8270005}
    J_1(t) \leq N \delta^{-\alpha} \bM_t\left( \|u(\star,\cdot)\|_{\Lambda^\gamma(\bR^d)} 1_{(0,s_k)}(\star) \right)(t).
\end{equation}
For $J_2$,
\begin{align*}
    J_2(t) &= N \sum_{j=0}^{\infty} \int_{t-2^{j+1}\delta}^{t-2^{j}\delta} |t-s|^{-1-\alpha} \|u(s,\cdot)\|_{\Lambda^\gamma(\bR^d)} 1_{(0,s_k)}(s) ds
    \\
    &\leq N \delta^{-1-\alpha} \sum_{j=0}^{\infty} 2^{-(1+\alpha)j} \int_{t-2^{j+1}\delta}^{t-2^{j}\delta} \|u(s,\cdot)\|_{\Lambda^\gamma(\bR^d)} 1_{(0,s_k)}(s) ds
    \\
    &\leq N \delta^{-\alpha} \bM_t\left( \|u(\star,\cdot)\|_{\Lambda^\gamma(\bR^d)} 1_{(0,s_k)}(\star) \right)(t) \sum_{j=0}^{\infty} 2^{-\alpha j}
    \\
    &= N \delta^{-\alpha} \bM_t\left( \|u(\star,\cdot)\|_{\Lambda^\gamma(\bR^d)} 1_{(0,s_k)}(\star) \right)(t).
\end{align*}
This and \eqref{eq8270005} yield
\begin{equation*}
\|h_k(t,\cdot)\|_{\Lambda^\gamma(\bR^d)} \leq N \delta^{-\alpha} \bM_t\big( \|u(\star,\cdot)\|_{\Lambda^\gamma(\bR^d)} 1_{(0,s_k)}(\star) \big)(t).
\end{equation*}
Therefore, by the weighted Hardy-Littlewood maximal function theorem,
\begin{equation} \label{eq8180118}
\|h_k\|_{\mathbf{\Lambda}_{p,\omega}^\gamma(s_{k-1},s_{k+1})} \leq N \delta^{-\alpha}\|u\|_{\mathbf{\Lambda}_{p,\omega}^\gamma(0,s_k)}.
\end{equation}
By Lemma \ref{lem8080013}, \eqref{eq8180023}, and \eqref{eq8180118},
\begin{align*}
\|u_k\|_{\mathbf{\Lambda}_{p,\omega}^\gamma(s_{k-1},s_{k+1})} &\leq N \delta^\alpha\|\partial_{t,s_{k-1}}^\alpha u_k\|_{\mathbf{\Lambda}_{p,\omega}^\gamma(s_{k-1},s_{k+1})}
\\
&\leq N \delta^\alpha \|f\eta_k\|_{\mathbf{\Lambda}_{p,\omega}^\gamma(s_{k-1},s_{k+1})} + N \delta^\alpha \|h_k\|_{\mathbf{\Lambda}_{p,\omega}^\gamma(s_{k-1},s_{k+1})}
\\
&\quad+ N \delta^\alpha \|u_k\|_{\mathbf{\Lambda}_{p,\omega}^\gamma(s_{k-1},s_{k+1})}
\\
&\leq N\delta^\alpha \|f\|_{\mathbf{\Lambda}_{p,\omega}^\gamma(T)} + N\|u\|_{\mathbf{\Lambda}_{p,\omega}^\gamma(0,s_k)} + N \delta^\alpha \|u_k\|_{\mathbf{\Lambda}_{p,\omega}^\gamma(s_{k-1},s_{k+1})}.
\end{align*}
Recall that $\delta=T/m$. Hence, by taking $m$ sufficiently large, one can absorb the last term of the right-hand side into the left-hand side, and
\begin{equation*}
    \|u_k\|_{\mathbf{\Lambda}_{p,\omega}^\gamma(s_{k-1},s_{k+1})} \leq N \|f\|_{\mathbf{\Lambda}_{p,\omega}^\gamma(T)} + N\|u\|_{\mathbf{\Lambda}_{p,\omega}^\gamma(0,s_k)}.
\end{equation*}
Since this inequality holds for $k=1,\cdots,m-1$, one can deduce that
\begin{equation*}
    \|u\|_{\mathbf{\Lambda}_{p,\omega}^\gamma(T)} \leq N(m) \big(\|f\|_{\mathbf{\Lambda}_{p,\omega}^\gamma(T)} + \|u\|_{\mathbf{\Lambda}_{p,\omega}^\gamma(0,s_1)}\big).
\end{equation*}
We further take $m$ sufficiently large to absorb the term involving $\|u\|_{\mathbf{\Lambda}_{p,\omega}^\gamma(0,s_1)}$ on the right-hand side of \eqref{eq8270027} into the left-hand side. Then we get $\|u\|_{\mathbf{\Lambda}_{p,\omega}^\gamma(T)} \leq N \|f\|_{\mathbf{\Lambda}_{p,\omega}^\gamma(T)}$.
The lemma is proved.
\end{proof}

\section{Proof of Theorem \ref{thm_main}}
\label{sec_solvability}

In this section, we prove our main result. Our strategy is to first establish solvability for the simple equation 
\begin{equation} \label{eq_simple}
    \partial_t^\alpha u(t,x) = \Delta u(t,x) + f(t,x)
\end{equation}
and then treat the general case using the method of continuity and the a priori estimate \eqref{est_apriori}.
For \eqref{eq_simple}, the representation formula for the solution is well known. Let $q(t,x)$ be the fractional heat forcing kernel defined by
\begin{equation*}
    \cF[q(t,\cdot)](\xi) = t^{\alpha-1} E_{\alpha,\alpha}(-t^\alpha |\xi|^2).
\end{equation*}
Moreover,
\begin{equation} \label{eq8281331}
    \|q(t,\cdot)\|_{L_1(\bR^d)} \leq Nt^{\alpha-1}
\end{equation}
and for every $T_0\in(0,T)$,
\begin{equation} \label{eq8281509}
    \int_{\bR^d} \sup_{t\in [T_0,T]} \left(|q(t,x)| + |\partial_t q(t,x)| \right) dx <\infty.
\end{equation}
We refer the reader to \cite[Lemma 3.7 and Corollary 3.9]{KPR21} for the existence of $q(t,x)$, \eqref{eq8281331}, and \eqref{eq8281509}.
It is well known that, for example, if $f\in C_c^\infty((0,T)\times\bR^d)$, then 
\begin{equation} \label{eq8281511}
    u(t,x) := \int_0^t \int_{\bR^d} q(t-s,x-y) f(s,y) dyds
\end{equation}
is a solution to \eqref{eq_simple} in the classical sense. In our setting, however, $f\in \mathbf{\Lambda}_{p,\omega}^{\gamma}(T)$, and strong measurability needs to be taken into account. We explain this point in more detail below.

\begin{lemma} \label{lem8281644}
    Let $T\in(0,\infty)$, $\alpha\in(0,1)$, $\gamma\in(0,\infty)$, $p\in(1,\infty]$. Suppose that $\omega\in A_p(\bR)$ with $[\omega]_{A_p(\bR)}\leq K_0$ (put $\omega\equiv1$ and $K_0=1$ when $p=\infty$). Let $f \in \mathbf{\Lambda}_{p,\omega}^\gamma(T)$.
Then there exists a solution $u \in \bH_{p,\omega,0}^{\alpha,\gamma+2}(T)$ to \eqref{eq_simple}.
\end{lemma}

\begin{proof}
We divide the proof into the cases $p\in(1,\infty)$ and $p=\infty$.
First, we assume that $p\in(1,\infty)$.
     As in the proof of Lemma \ref{lem_dense}, $f$ can be approximated in $\mathbf{\Lambda}_{p,\omega}^\gamma(T)$ by functions in $C^1([0,T];\Lambda^\gamma(\bR^d))$. Thus, due to \eqref{est_apriori}, it suffices to prove the lemma when $f\in C^1([0,T];\Lambda^\gamma(\bR^d))$.

Let $u$ be defined by \eqref{eq8281511}.
     First, we show that $u$ is strongly measurable. Let
     \begin{equation*}
         v(t,s) :=\int_{\bR^d} q(t-s,x-y) f(s,y) dy.
     \end{equation*}
     Then by \eqref{eq8281509}, for $t_1\in(s_1+\delta,T)$ and $t_2\in (s_2+\delta,T)$ with $\delta>0$,
     \begin{align*}
         &\|v(t_1,s_1)-v(t_2,s_2)\|_{\Lambda^\gamma(\bR^d)} 
         \\
         &\leq N(\delta,T) \left(|t_1-t_2|+|s_1-s_2|\right) \sup_{s\in(0,T)}\left(\|f(s)\|_{\Lambda^\gamma(\bR^d)} + \|\partial_s f(s)\|_{\Lambda^\gamma(\bR^d)}\right).
     \end{align*}
     Thus, $v(t,s)$ is continuous on $\{(t,s): 0\leq s<t\leq T\}$ with values in $\Lambda^\gamma(\bR^d)$, which yields the measurability. By \eqref{eq8281331}, for $t\in(s,T)$,
     \begin{equation*}
         \|v(t,s)\|_{\Lambda^\gamma(\bR^d)} \leq N(t-s)^{\alpha-1} \|f(s)\|_{\Lambda^\gamma(\bR^d)}.
     \end{equation*}
     Thus, by Fubini's theorem, one can deduce that $u$ is also strongly measurable.
Moreover, by \eqref{eq9231819},
     \begin{equation*}
         \|u(t,\cdot)\|_{\Lambda^{\gamma}(\bR^d)} \leq N(T) \bM_t\left( \|f(\star,\cdot)\|_{\Lambda^\gamma(\bR^d)} 1_{(0,T)}(\star) \right)(t),
     \end{equation*}
     which together with the weighted Hardy-Littlewood maximal function theorem implies that $u\in \mathbf{\Lambda}_{p,\omega}^{\gamma}(T)$.

     For each $t\in(0,T)$,
     \begin{equation*}
         \cF[\Delta_ku(t,\cdot)](\xi)  = \int_0^t (t-s)^{\alpha-1} E_{\alpha,\alpha}(-(t-s)^\alpha|\xi|^2) \cF[\Delta_kf(s,\cdot)](\xi) ds,
     \end{equation*}
     which implies that
     \begin{equation*}
         \Delta_k u(t,x) = \int_0^t \int_{\bR^d} q(t-s,x-y) \Delta_k f(s,y) dyds.
     \end{equation*}
     Here, $\Delta_k f$ is infinitely differentiable in $x$, with bounded derivatives. Therefore, we can apply \cite[Theorem 1.1]{CKKW20} to obtain that for any $t\in(0,T]$,
    \begin{equation} \label{eq9231838}
        I^{1-\alpha}_0(\Delta_ku)(t) = \int_0^t \left( \Delta (\Delta_ku)(s) + \Delta_kf(s) \right) ds
    \end{equation}
    in the space of bounded uniformly continuous functions.
More precisely, the notation in the reference can be applied with the following choices:
\begin{itemize}
\item for the temporal part,
\begin{equation*}
w(t)=\frac{t^{-\alpha}}{\Gamma(1-\alpha)}, \qquad \nu(ds)=\frac{\alpha}{\Gamma(1-\alpha)}s^{-1-\alpha}\,ds,
\end{equation*}
and $S_r$ is the $\alpha$-stable subordinator satisfying $\bE[e^{-\lambda S_r}]=e^{-r\lambda^\alpha}$,
\item for the spatial part,
\begin{equation*}
(E,\mu(dx))=(\bR^d,dx), \qquad P_r^0=e^{r\Delta}, \qquad \cL=\Delta,
\end{equation*}
and $\bB$ is the space of bounded uniformly continuous functions.
\end{itemize}

Let us consider the time mollification given by \eqref{eq8252326}.
We claim that $u^{(\varepsilon)}\in \bH_{p,\omega,0}^{\alpha,\gamma+2}(T)$ and 
\begin{equation} \label{eq9231854}
        \|u^{(\varepsilon)}\|_{\bH_{p,\omega,0}^{\alpha,\gamma+2}(T)} \leq N \|f^{(\varepsilon)}\|_{\mathbf{\Lambda}_{p,\omega}^{\gamma}(T)}.
\end{equation}
First, we show that $I_0^{1-\alpha}u^{(\varepsilon)}$ is a $\Lambda^{\gamma}(\bR^d)$-valued absolutely continuous function such that $(I_0^{1-\alpha}u^{(\varepsilon)})(0,\cdot)=0$. Since $u \in \mathbf{\Lambda}_{p,\omega}^{\gamma}(T)$, as in Lemma \ref{lem_dense}, $u^{(\varepsilon)}$ is infinitely differentiable in $t$ with values in $\Lambda^{\gamma}(\bR^d)$, and $u^{(\varepsilon)}(0,\cdot)=0$. Thus, due to \eqref{eq8241804}, $I_0^{1-\alpha}u^{(\varepsilon)} = I_0^{2-\alpha}\frac{d}{dt}u^{(\varepsilon)}$, which leads to the desired result.

Since the existence of $\partial_t^\alpha (u^{(\varepsilon)}) \in \mathbf{\Lambda}_{p,\omega}^{\gamma}(T)$ easily follows from the differentiability of $u^{(\varepsilon)}$, to conclude $u^{(\varepsilon)} \in \bH_{p,\omega,0}^{\alpha,\gamma+2}(T)$, it remains to show that $u^{(\varepsilon)} \in \mathbf{\Lambda}_{p,\omega}^{\gamma+2}(T)$.
Due to \eqref{eq9231838}, $\Delta_k u^{(\varepsilon)} \in C^1([0,T];X_k)$ and is a pointwise solution to
\begin{equation*}
    \Delta_k \left(\partial_t^\alpha (u^{(\varepsilon)})\right) = \Delta_k \left((\Delta u)^{(\varepsilon)}\right) + \Delta_k f^{(\varepsilon)}, \quad (t,x)\in (0,T)\times\bR^d.
\end{equation*}
Thus, for each $t\in(0,T)$,
\begin{equation} \label{eq9231902}
    \Delta u^{(\varepsilon)} = \partial_t^\alpha (u^{(\varepsilon)}) - f^{(\varepsilon)}
\end{equation}
in the sense of tempered distributions. This and $\partial_t^\alpha (u^{(\varepsilon)}) \in \mathbf{\Lambda}_{p,\omega}^{\gamma}(T)$ imply that $\Delta u^{(\varepsilon)} \in \mathbf{\Lambda}_{p,\omega}^{\gamma}(T)$. Since $(1-\Delta)^{-1}: \Lambda^{\gamma}(\bR^d)\to \Lambda^{\gamma+2}(\bR^d)$ is a bounded isomorphism, the relation
\begin{equation*}
    u^{(\varepsilon)} = (1-\Delta)^{-1}\left[ u^{(\varepsilon)} - \Delta u^{(\varepsilon)} \right],
\end{equation*}
leads to $u^{(\varepsilon)} \in \mathbf{\Lambda}_{p,\omega}^{\gamma+2}(T)$, which in fact implies that $u^{(\varepsilon)} \in \bH_{p,\omega,0}^{\alpha,\gamma+2}(T)$.

Now we prove \eqref{eq9231854}. Due to \eqref{eq9231902}, one can use \eqref{eq8171625} to obtain that
\begin{align*}
        \|u^{(\varepsilon)}(t,\cdot)\|_{\Lambda^{\gamma+2}(\bR^d)} \leq N \bM_t\left(\|f^{(\varepsilon)}(\star,\cdot)\|_{\Lambda^{\gamma}(\bR^d)} 1_{(0,T)}(\star)\right)(t).
\end{align*}
Due to the strong measurability of $u^{(\varepsilon)}(t,\cdot)$, we can apply the weighted Hardy-Littlewood maximal function theorem to get \eqref{eq9231854}.

Let us now return to showing that $u$ is indeed a solution.
When $p<\infty$, since $f^{(\varepsilon)}\to f$ in $\mathbf{\Lambda}_{p,\omega}^{\gamma}(T)$, by \eqref{est_apriori}, $u^{(1/n)}$ is Cauchy in $\bH_{p,\omega,0}^{\alpha,\gamma+2}(T)$. Thus, one can deduce that $u\in \bH_{p,\omega,0}^{\alpha,\gamma+2}(T)$. Moreover, by \eqref{eq9231902}, $u$ is a solution to \eqref{eq_simple}.

Lastly, we deal with $p=\infty$. Note that $\mathbf{\Lambda}_{\infty}^{\gamma}(T) \subset \mathbf{\Lambda}_{q,1}^{\gamma}(T)$ for any $q\in(1,\infty)$. Thus, by the above result, there is a solution $u\in \bH_{q,1,0}^{\alpha,\gamma+2}(T)$ to \eqref{eq_simple}. Then by \eqref{eq8171625},
\begin{align*}
    \|u(t,\cdot)\|_{\Lambda^{\gamma+2}(\bR^d)} \leq N\bM_t\left( \|f(\star,\cdot)\|_{\Lambda^{\gamma}(\bR^d)}1_{(0,T)}(\star) \right)(t) \leq N \|f\|_{\mathbf{\Lambda}_{\infty}^{\gamma}(T)},
\end{align*}
which easily yields $u\in \bH_{\infty,0}^{\alpha,\gamma+2}(T)$. We notice that the measurability follows from the case $p\in(1,\infty)$.
The lemma is proved.
\end{proof}

\begin{proof}[Proof of Theorem \ref{thm_main}]
Note that the estimate \eqref{eq_main_est} follows from Lemma \ref{lem_apriori}. Thus, we only need to use the a priori estimate together with Lemma \ref{lem8281644}.
The theorem is proved.
\end{proof}

\vspace{3em}
\textbf{Declaration of interest}
The authors have no competing interests to declare that are relevant to the content of this article.

\textbf{Data availability}
Data Availability No datasets were generated or analysed during the current study.


\end{document}